\documentclass[final]{siamltex}

\usepackage{amsmath,amssymb,float,graphicx,multirow,subfig,bigstrut}
\usepackage[notref, notcite, color]{showkeys}
\usepackage{multirow,cancel,soul}
\def\Xint#1{\mathchoice
{\XXint\displaystyle\textstyle{#1}}%
{\XXint\textstyle\scriptstyle{#1}}%
{\XXint\scriptstyle\scriptscriptstyle{#1}}%
{\XXint\scriptscriptstyle\scriptscriptstyle{#1}}%
\!\int}
\def\XXint#1#2#3{{\setbox0=\hbox{$#1{#2#3}{\int}$ }
\vcenter{\hbox{$#2#3$ }}\kern-.6\wd0}}

\def\dashint{\Xint-}

\def\bsymbol#1{\mbox{\boldmath$\displaystyle#1$\unboldmath}}

\newcommand{\Mjj}{{M_{j}}}
\newcommand{\ud}{\ \mathrm{d}}
\newcommand{\bx}{\mathbf{x}}
\newcommand{\bu}{\mathbf{u}}
\newcommand{\bv}{\mathbf{v}}
\newcommand{\bm}{{\bf Q}}

\newcommand{\bG}{{\cal G}}
\newcommand{\bs}{\boldsymbol}
\newcommand{\bfmu}{{\bsymbol{\mu}}}
\newcommand{\bfchi}{{\bsymbol{\chi}}}
\newcommand{\bfvarphi}{{\bsymbol{\chi}}}
\newcommand{\bflambda}{{\bsymbol{\lambda}}}

\newtheorem{remark}{Remark}[section]
\newcommand{\pdiff}[2]{\frac{\partial {#1}}{\partial {#2}}}
\renewcommand{\epsilon}{\varepsilon}
\newcommand{\eps}{\varepsilon}

\newcommand{\xref}[1]{(\ref{#1})}

\title{Preconditioning nonlocal multi-phase flow}
\author{\textsc{David Kay\footnotemark[1] \and Vanessa Styles\footnotemark[2]}}

\begin{document}
\maketitle



\begin{abstract}
We propose an efficient solver for saddle point problems arising from finite element approximations of nonlocal multi-phase Allen--Cahn
variational inequalities. The solver is seen to behave mesh independently and to have only a very mild dependence on the number of phase field variables. 
In addition we prove convergence, in three GMRES iterations, of the approximation of the two phase problem, regardless of mesh size or interfacial width. 
Numerical results are presented that illustrate the competitiveness of this approach.
\end{abstract}



\let\oldthefootnote\thefootnote
\renewcommand{\thefootnote}{\fnsymbol{footnote}}
\footnotetext[1]{Department of Computer Science, University of Oxford,  Oxford, OX1 3QD,  UK \\ david.kay@cs.ox.ac.uk}
\footnotetext[2]{Department of Mathematics, University of Sussex, Brighton, BN1 9RF, UK \\ v.styles@sussex.ac.uk}
\let\thefootnote\oldthefootnote

\pagestyle{myheadings}
\thispagestyle{plain}
\markboth{D.~KAY and V. STYLES}{Preconditioning nonlocal multi-phase flow}

\section{Introduction}
\label{s:intro}
The aim of this paper is to combine preconditioning methods for indefinite problems and multigrid preconditioning developed for elliptic
systems to provide an efficient preconditioner for the solution of 
systems of multiphase Allen--Cahn variational inequalities of the form: 


${\cal P}_\bm$: For given $\bu(\cdot,0) = \bu_0 \in {\cal G}_{\bm}$, find $\bu \in L^2 (0,T;{\cal G}_{\bm}) \cap H^1(0,T; \mathbf{L}^2 (\Omega))$ such that 
$$
\epsilon \left(\pdiff{\bu}{t}, \bfchi - \bu \right) + \epsilon \left( \nabla \bu, \nabla \left( \bfchi-\bu \right) \right) - \frac{1}{\epsilon} \left (\mathbf{A} \bu, \bfchi - \bu \right) \geq 0, \forall~\bfchi \in {\cal G}_\bm.
$$

Here $\Omega \in \mathbb{R}^d$, $d=1,2$ or $3$ and $\bu:\Omega \times (0,T) \to\mathbb{R}^N$  denotes the vector-valued phase field function 
which describes the fractions of the $N$ phases, i.e. each component of $\bu$ describes one phase, $A \in \mathbb{R}^{N \times N}$ is a symmetric constant matrix that has at least one positive eigenvalue, 
$$
\mathcal{G}_\bm :=\{\mathbf{v}\in \mathbf{H}^1(\Omega)\mid \mathbf{v}\in G_\bm ~~ a.e.\} 
~~\mbox{with}~~
G_\bm:= \{\bsymbol{\xi} \in \mathbb{R}^N\mid \bsymbol{\xi} \geq \bsymbol{0}, \bsymbol{\xi} \cdot \mathbf{1} = 1, \mbox{$\dashint_\Omega \bv \ud\bx$}= \bm\}. 
$$
We denote by $\mathbf{L}^2(\Omega)$
and $\mathbf{H}^1(\Omega)$ the spaces of vector-valued functions, 
$(.,.)$ is the standard $L^2$ inner product for scalar functions, $(\mathbf{v}, \mathbf{w}) =
\sum \limits_{i=1}^N (v_i, w_i)$ for $\mathbf{v}, \mathbf{w} \in
\mathbf{L}^2(\Omega)$, $(\mathbf{A},\mathbf{B}) =
\sum\limits^N_{i=1} \sum\limits^d_{j=1}(a_{ij},b_{ij})$ for
matrix-valued functions, $\bsymbol{\xi} \geq \mathbf{0}$ means $\xi_i \geq 0$ for all $i \in \{1,...,N\}$, $\mathbf{1} =
(1,...,1)^T$, $\bsymbol{\xi} \cdot \mathbf{1} = \sum \limits_{i=1}^N \xi_i$ and ${\dashint_\Omega f(x)\ud\bx  :=\tfrac{1}{|\Omega|} \int_\Omega f(x)\ud\bx}$.

The system ${\cal P}_\bm$  arises from steepest descent dynamics with respect to the $L^2$-norm of  the Ginzburg-Landau energy, 
$$ 
E(\bu) := \int_{\Omega} \left( \frac{\epsilon}{2} \left| \nabla \bu \right|^2 + \frac{1}{\epsilon} \Psi(\bu) \right)\! \ud \bx,
$$
under the constraint $\dashint_\Omega \bu \ud \bx = \bm$.
Here $\Psi$ is the multi-obstacle potential 
$$ \Psi(\theta) :=
-\frac{1}{2} \theta^T \mathbf{A} \theta+ I_\bG =  \left\{ \begin{array}{ll} 
-\frac{1}{2} \theta^T \mathbf{A} \theta & \mbox{ for } \theta \in \bG \\
\infty & \mbox{ otherwise}
\end{array}\right.
$$
with $I_\bG$ denoting the indicator function for the Gibbs Simplex, 
$G:= \{\bsymbol{\xi} \in \mathbb{R}^N\mid \bsymbol{\xi} \geq \bsymbol{0}, \bsymbol{\xi} \cdot \mathbf{1} = 1 \}$, and the symmetric constant matrix $A$ has at least one positive eigenvalue to allow for minima of $\Psi$ to exist, see \cite{ElliottLuckhaus}. 

\vspace{2mm}

{
\begin{remark}
Steepest descent dynamics with respect to the $L^2$-norm of the Ginzburg-Landau energy $E(\bu)$, without the constraint $\dashint_\Omega \bu \ud \bx = \bm$,  yields the system:\\ 
${\cal P}$: For given $\bu(\cdot,0) = \bu_0 \in {\cal G}$, find $\bu \in L^2(0,T;{\cal G}) \cap H^1(0,T; \mathbf{L}^2 (\Omega))$ such that
$$
\epsilon \left(\pdiff{\bu}{t}, \bfchi - \bu \right) + \epsilon \left( \nabla \bu, \nabla \left( \bfchi-\bu \right) \right) - \frac{1}{\epsilon} \left (\mathbf{A} \bu, \bfchi - \bu \right) \geq 0,~~~\forall~\bfchi \in {\cal G}
$$
where $\mathcal{G} :=\{\mathbf{v}\in \mathbf{H}^1(\Omega)\mid \mathbf{v}\in G ~~ a.e.\}$.
Since ${\cal P}$ is a simplified version of ${\cal P}_{\bm}$ the solver we propose in this paper can be applied to the corresponding finite element approximation of ${\cal P}$. 
\end{remark}
}

\vspace{2mm}

${\cal P}$ is a generalisation of the scalar Allen--Cahn equation that was introduced by Allen and Cahn \cite{AllenCahn1979} to describe the capillarity driven
evolution of an interface separating two bulk phases. The parameter $\eps$, with $0<\eps \ll 1$, is associated with the thickness of the diffuse interfacial layer in which 
the phase field variables rapidly change their value. The $N$ phase extension of the scalar Allen--Cahn model was introduced in \cite{BR, GNS1999}. 
The nonlocal problem ${\cal P}_\bm$ models interface evolution with mass conservation.

Multiphase Allen--Cahn models have a variety of applications, 
including image segmentation, see for example \cite{TK2009}, 
and identification of coefficients in elliptic equations \cite{DES}.  
Applications arising from identification of coefficients in elliptic equations include electric impedance tomography, \cite{DorVil08, IglMcL11, IglLinStu14}, and 
flow in porous media with unknown permeabilities \cite{CheIsaNew99, DorMilRap00, BoyAdlLio12}. 
Applications of mass conserving multiphase Allen--Cahn models include structural topology optimisation \cite{WZ,BFGS}.



Efficient and reliable, i.e., fast and globally converging, multigrid methods for solving implicit in time finite element approximations of ${\cal P}$ are presented in \cite{KK2,KK2006}, while an explicit in time finite element approximation of ${\cal P}_\bm$ was introduced in \cite{GNSW}. In \cite{BlankGarkeSarbuStyles2013} (semi-)implicit in time finite element approximations of ${\cal P}_\bm$ are considered in which a primal-dual active set method, see \cite{BIK,HIK}, is used to solve the finite element approximations. 
By using Krylov-subspace solvers and suitable preconditioners the authors in \cite{blanksarbustoll2012} develop 
efficient, mesh independent, solvers for the (semi-)implicit approximations of ${\cal P}$ and ${\cal P}_\bm$ that were derived in \cite{BlankGarkeSarbuStyles2013}. In this work we introduce an alternative preconditioner to the ones in \cite{blanksarbustoll2012} resulting in a solver that is not only mesh independent, but also is only mildly dependent of the number of phases $N$. 


We note that in \cite{GKCH} globally convergent nonsmooth Schur--Newton methods are introduced for the solution of discrete multicomponent Cahn--Hilliard systems with logarithmic and obstacle potentials. These methods could also be used to solve the multicomponent Allen--Cahn systems ${\cal P}$ and ${\cal P}_\bm$.


{When using iterative techniques to solve the linear system that arises when the primal-dual active set method is used to solve a finite element approximation of ${\cal P}_\bm$, 
the majority of the work to be undertaken within each iteration is in the solving of the linear systems $K \bx= \bs b$} 
$$ 
K=\left[\begin{array}{cc}
{\cal K} & B^T\\
B&0
\end{array}\right],
$$ 
where ${\cal K}$ is symmetric positive definite. Similar saddle point structures are common place within fluid dynamics, leading to much development of numerical solvers for Navier-Stokes equations, e.g. see \cite{murphy2000note,kay2002preconditioner,ElmanSilvesterwathenbook}. In these papers, it is the choice of preconditioning matrix, $P$, that leads to improved  convergence of the chosen iterative Krylov subspace scheme, e.g. see \cite{saad1996iterative}.

The preconditioning of the linear systems arising from ${\cal P}$ and ${\cal P}_\bm$ 
was initially considered in 
\cite{blanksarbustoll2012}, where a preconditioning technique building upon Stokes type systems is proposed. In this work we will use a similar structural approach to that of \cite{blanksarbustoll2012}, however, we provide an improved approximation to the Schur complement, $B{\cal K}^{-1}B^T$, leading to a solver that is almost independent of the number of phases. Moreover, this improvement does not lead to any significant increase in computational effort per iteration, ultimately leading to a more effective solver. In addition, when considering the two phase problem, with $N=2$, the minimal polynomial of the resulting preconditioned system is of degree three and hence GMRES will converge within three iterations, see Theorem \ref{theorem}.

The paper is organised as follows. In Section \ref{s:fea} we reformulate 
${\cal P}_\bm$ with the help of Lagrange multipliers, yielding the associated system 
${\cal R}_\bm$. We then introduce a finite element approximation of an implicit Euler-discretisation of ${\cal R}_\bm$ and we apply a primal-dual active set algorithm to this discretisation. In Section \ref{s:precon} a preconditioner for the primal-dual active set algorithm is developed and the implementation of the numerical solver is presented.
In Section \ref{s:results} we present numerical computations that illustrate the effectiveness of our approach, in particular they show how the iteration number is independent of the mesh size and only mildly dependent on the number of phases $N$.

\section{Alternative Formulation and Finite Element Discretisation}\label{s:fea}
In this section we follow the authors in \cite{BlankGarkeSarbuStyles2013} in applying a primal-dual active set
method, \cite{BIK,HIK}, to a finite element approximation of ${\cal P}_\bm$, this method is well known in the context of optimisation with partial differential
equations as constraints.  To this end we first reformulate  ${\cal P}_\bm$ with the help of Lagrange multipliers, yielding the associated system 
${\cal R}_\bm$,  then we apply a Primal Dual Active Set algorithm to a finite element approximation of ${\cal R}_\bm$.

\subsection{Alternative Formulation of ${\cal P}_\bm$}
In \cite{BlankGarkeSarbuStyles2013} the following alternative formulation of 
${\cal P}_\bm$ is presented:

${\cal R}_\bm$: Let $\Omega \subset \mathbb{R}^d$ be a bounded domain which is 
 either convex or fulfills $\partial  \Omega \in C^{1,1}$.  For given $\bu(\cdot,0) = \bu_0 \in {\cal G}$, find  $\bu \in L^2 \left(0,T; \mathbf{H}^2 (\Omega) \right) \cap H^1(0,T; \mathbf{L}^2 (\Omega))$, $\bfmu \in L^2(0,T; \mathbf{L}^2 (\Omega)) $, $\bflambda \in L^2(0,T;S)$ and $\Lambda \in  L^2(0,T; L^2(\Omega)) $ such that 
\begin{eqnarray*}\label{eq: lagrange system}
	\epsilon\pdiff{\bu}{t} - \epsilon \Delta \bu - \frac{1}{\epsilon}\mathbf{A} \bu - \frac{1}{\epsilon} \bfmu- \frac{1}{\epsilon} \Lambda {\bf 1} -\frac{1}{\epsilon} \bflambda &= &{\bf 0}  \mbox{ a.e. in } \Omega \times (0,T), \\
 	\bu \cdot {\bf 1} & = & 1 \mbox{ a.e. in } \Omega \times (0,T), \\
	 (\bu, 1) = \bm, ~ (\bfmu,\bu) &=& 0, \mbox{ for almost all } t \in (0,T), \\
 	\bu \geq \bs 0,~\bfmu & \geq & \bs 0, \mbox{ a.e. in } \Omega \times (0,T).
 \end{eqnarray*}
Here the Lagrange multipliers $\bfmu, \Lambda$ and $\bflambda$ are such that 
\begin{itemize}
\item[(i)] $\bfmu(\bx, t): \Omega \times (0,T) \rightarrow \mathbb{R}^N$, is used to impose the 
constraint $\bu \geq \bs 0$,

\item[(ii)]  $\bflambda(t):  \mathbb{R}^N \times (0,T) \rightarrow S$, is used to impose the mass constraint $\dashint_{\Omega} \bu ~ d \bx = \bm$,

\item[(iii)] $\Lambda(\bx,t): \Omega \times (0,T) \rightarrow \mathbb{R}$, is used to impose the saturation constraint $\bu \cdot {\bf 1} = 1$,
\end{itemize}
and
$S := \left\{ \bv \in \mathbb{R}^N: \bv \cdot {\bf 1} = 0 \right\}\!$.
 
${\cal R}_{\bm}$ 
is complemented with the 
the boundary condition $\frac{\partial \bu}{\partial \nu}=0$, were $\nu$ is the outer unit normal to $\partial\Omega$. 

\begin{remark} 
The scaling $\frac1{\varepsilon}\bfmu$ in ${\cal R}_{\bm}$ is introduced in order that $\bfmu$ is of order one, if we were to replace $\frac1{\varepsilon}\bfmu$ by 
$\bfmu$ we would observe a severe $\varepsilon$-dependence of $\bfmu$ which in practice often results in
oscillations in the discretised primal-dual active set method.  
\end{remark}

\subsection{Finite Element Discretisation}
For simplicity we assume that $\Omega$ is a polyhedral domain. Let $\mathcal{T}_h$ be a 
regular triangulation of $\Omega$ into disjoint open simplices,
i.e.\ in particular $\overline{\Omega} = \cup_{T\in \mathcal{T}_h} \overline{T}$.
Furthermore, we define $h:=\max_{T\in \mathcal{T}_h}\{\mbox{diam} ~T\}$ the maximal element size of $\mathcal{T}_h$
and we set $\mathcal{J}$ to be the set of nodes of 
$\mathcal{T}_h$ and $\{\mathbf{p}_j\}_{j\in \mathcal{J}}$ 
to be the coordinates of these nodes. 
Associated with $\mathcal{T}_h$ is the piecewise linear finite element space
$$
S_h:=  \Big\{\varphi\in C^{0}(\overline\Omega)
\Big| \,  \varphi_{\big|_{T}}\in P_{1}(T) \,\,\,\forall ~T\in \mathcal{T}_h \Big\}\subset H^1(\Omega),
$$
where we denote by $P_{1}(T)$ the set of all affine linear functions on $T$. 
Furthermore we denote the standard nodal 
basis functions of $S_h$ by $\{\chi_j\}_{j\in\mathcal{J}}$ 
and we set $\mathbf{S}_h=(S_h)^N$.

The time domain $(0,T)$ is divided into $N_T$ uniform intervals $(t_{n-1}, t_{n})$, with $\tau := t_n-t_{n-1}$, $n=1,2,\ldots, N_T$.  
For simplicity of presentation we denote by $\bu_h\in \mathbf{S}_h$ the discrete solution at time $t_n$, while the solution at the previous time step will be denoted by $\bu_h^{n-1}\in \mathbf{S}_h$, and similarly for $\bfmu_h$,  $\bflambda_h$ and $\Lambda_h$.

We consider the following fully discrete approximation of ${\cal R}_\bm$:

${\cal R}_{\bm}^{h,\tau}$
{\it Given $\bu_h^{n-1}\in \mathbf{S}_h$, find $\bu_h\in \mathbf{S}_h$, $\bfmu_h\in \mathbf{S}_h$, $\bflambda_h \in\mathbb{R}^N$ and
  $\Lambda_h \in S_h$ such that
\begin{align}
&
\hspace*{-3mm}
\begin{array}{r}
~~~~~~~~~\tfrac{\eps^2}{\tau} (\bu_h ,\bfvarphi)_h 
+ \eps^2 (\nabla \bu_h ,\nabla\bfvarphi) 
- (\bfmu_h,\bfvarphi )_h - (\bflambda_h,\bfvarphi) - (\Lambda_h \mathbf{1},\bfvarphi)_h ~~~~~~~~~~\\
 ~~~~~~~~~
= \tfrac{\eps^2}{\tau} (\bu_h^{n-1},\bfvarphi)_h +  (\mathbf{A}\bu_h^{n-1}, \bfvarphi)_h 
~~\forall\bfvarphi\in \mathbf{S}_h,
\end{array} \label{eq2a}
\\
&\hspace*{2cm}
\sum_{i=1}^N (u_i)_j = 1 ~~~ \forall \; j \in \mathcal{J},
 \label{eq2b1}\\
&\hspace*{2cm}
\sum \limits_{j \in \mathcal{J}} \Mjj((u_i)_j-(u_N)_j) = 
\sum \limits_{j \in \mathcal{J}} \Mjj (m_i - m_N)\quad\mbox{for}\quad i
\in \{1,...,N-1\},
\label{eq2b}\\
&\hspace*{2cm}
\lambda_N = -\sum^{N-1}_{i=1}\lambda_i,\label{eq2lam}\\
& \hspace*{2cm}
\bfmu_j\geq \mathbf{0},~{\bu}_j\geq \mathbf{0}~~~\forall~j\in\mathcal{J},
 ~~~~
 ({\bu}_h, \bfmu_h)_h =0.
\label{eq:cons2a}
\end{align}
}
Here $(u_i)_j$ denotes the $i$-th component $u_i$ of
$\bu$ at the $j$-th node, $(f,g)_h = \int_\Omega I_h(f g )$ denotes the lumped mass semi-inner product where 
$I_h : C^{0}(\overline\Omega) \to S_h$ is the standard interpolation
operator such that $(I_h\,f)(\mathbf{p}_j)=f(\mathbf{p}_j)$ for all nodes $j\in \mathcal{J}$ and $\Mjj:=(\chi_j,\chi_j)_h$, $j\in\mathcal{J}$.\\

\begin{remark} 
Due to the $\mathbf{A} \bu$ term in ${\cal P}_{\bm}$, the problem is non-convex, in the above discretisation we have chosen to treat this term fully explicitly.  
Alternative choices would be to treat this term fully implicitly, or in a semi--implicit manner, neither of which would affect the performance of our proposed solver. 
\end{remark}

\subsection{The Primal Dual Active Set Method}

\vspace{2mm}

We use the nonlinear primal dual active set (PDAS) algorithm derived in \cite{BlankGarkeSarbuStyles2013} to solve ${\cal R}_{\bm}^{h,\tau}$. 
The algorithm is obtained by reformulating the complementarity conditions (\ref{eq:cons2a}) using active sets
based on the primal variable $\mathbf{u}$ and the dual variable $\bfmu$. 
Here we use the notation $(u_i^k)_j$ and $(u_i^{n-1})_j$
where $k$ denotes the $k$-th iteration in the PDAS algorithm and $n-1$ is
the $(n-1)$-st time step. This is of course a misuse of notation for $k=n-1$. In addition for $c>0$, we set 
$$
\mathcal{A}_i^{k+1}:=\{ j \in \mathcal{J} : c \cdot (u^{k}_i)_j - {(\mu^{k}_i)_j} < 0 \},
$$
{
and we define the lumped mass diagonal matrix $M := (m_{ij})_{i,j\in\mathcal{J}}$ with $m_{ij}=(\chi_i,\chi_j)_h$ and 
the stiffness matrix $L := (l_{ij})_{i,j\in\mathcal{J}}$  with $l_{ij}=(\nabla \chi_j, \nabla \chi_i)$.
We define the mass lumped vector ${\bf m} := (\Mjj)_{j\in \mathcal{J}}$, 
and the entries of the matrix $\mathbf{A}$ by $a_{ij}$, $i,j=\{1,\dots,N\}$.} 

~~\\

{\bf Primal-Dual Active Set Algorithm (PDAS):}
\begin{enumerate}
\item[0.]
Set $c=\frac{2}{h^2}0$, $k=0$ and initialise $\mathcal{A}^0_i\subset\mathcal{J}$ for all $i \in \{1,..., N\}$.

\item[1.] 
Define 
$\mathcal{I}_i^{k}=\mathcal{J}\setminus \mathcal{A}^{k}_i$ for all $i \in \{1,...,N\}$.\\
Set $(u^{k}_i)_j= 0$ for  $j \in \mathcal{A}^k_i$, $(u^{k}_i)_j= 1$ for  $j\in \mathcal{I}^k_i\setminus \mathcal{D}^k_i$ and  
$(\mu^k_i)_j=0$ for $j\in \mathcal{I}_i^k$ for all $i \in \{1,...,N\}$.

\item[2.]
Set $\mathcal{D}^k_i:={\cal I}^k_i\cap(\bigcup\limits^N_{j=1\atop j\ne i}{\cal I}^k_j)$, $\mathcal{D}^k:=\bigcup\limits^N_{i=1}
\mathcal{D}^k_i$. Solve the discretised 
PDE (\ref{eq2a}) on the interface $\mathcal{D}^k$ with the constraints
(\ref{eq2b1})-(\ref{eq2lam}) to obtain
$(u^k_i)_j$ for all $(i,j)$ such that $j \in \mathcal{D}^k_i, i \in \{1,..., N\}$,  $\Lambda_j^k$ for all $j \in \mathcal{D}^k$, and $\lambda_i^k$ for all $i \in \{1,...,N\}$. More precisely we solve

\begin{align}
&\tfrac{\varepsilon^2}{\tau} (u_i^{k})_j 
+ \tfrac{ \varepsilon^2}{\Mjj} \sum_{r\in \mathcal{J}} l_{rj}(u_i^{k})_r 
-  \lambda_i^{k}-\Lambda_{j}^k
= \tfrac{\varepsilon^2}{\tau} (u_i^{n-1})_j 
+  \sum_{m=1}^N a_{im}(u^{n-1}_m)_j\ \nonumber\\
& \hspace*{5.5cm}\textrm{ for } j\in\mathcal{D}^k_i ~~\textrm{ and } i\in\{1,\dots,N\}\,,\label{pdas1}\\[2mm]
& \sum_{i=1}^N (u_i^k)_j = 1 ~ \textrm{ for } j\in\mathcal{D}^k, \label{pdas3}\\
&\sum_{j\in \mathcal{J}}\Mjj ( (u_i^{k})_j - (u_N^k)_j)
=  \sum_{j\in  \mathcal{J}}\Mjj (m_i - m_N) ~ \textrm{ for } i \in \{1,...,N-1\}\,, \label{pdas2}
\end{align}
where $\lambda_N^k = -\sum_{i=1}^{N-1}\lambda_i^k$ 
is used in (\ref{pdas1}).

\item[3.] Define $\Lambda^k_j$ for all $j\in {\cal I}^k_i\setminus\mathcal{D}^k$ as 
\begin{eqnarray*}
\Lambda^k_j=\frac{\varepsilon^2}{\tau}(u^k_i)_j
+\frac{\varepsilon^2}{\Mjj}\sum_{r\in\mathcal{J}}
l_{rj}(u^k_i)_r-\lambda^k_i
-\frac{\varepsilon^2}{\tau}(u^{n-1}_i)_j+\sum^N_{m=1}a_{im}(u^{n-1}_m)_j\,.
\end{eqnarray*}

\item[4.] Determine $(\mu^k_i)_j $ for $j \in \mathcal{A}_i^k$ using (\ref{eq2a}) for all
$i=1,...,N$ as
\begin{eqnarray*}
(\mu^k_i)_j= 
\tfrac{\varepsilon^2}{\tau}(u_i^{k})_j 
 +  \tfrac{\varepsilon^2}{\Mjj} \sum_{r\in \mathcal{J}} l_{rj}(u_i^{k})_r 
 - \lambda^{k}_i- \Lambda^k_j 
 - \tfrac{\varepsilon^2}{\tau} (u_i^{n-1})_j+ \sum_{m=1}^N a_{im}(u_m^{n-1})_j\,.
\end{eqnarray*}

\item[5.] Set
$\mathcal{A}_i^{k+1}:=\{ j \in \mathcal{J} : c (u^{k}_i)_j - {(\mu^{k}_i)_j} < 0 \}$ for 
$i \in\{1,..., N\}$.

\item[6.] 
If $\mathcal{A}_i^{k+1}=\mathcal{A}_i^k$ for all $i \in \{1,...,N\}$ stop, 
otherwise set $k=k+1$ and goto 1.
\end{enumerate}

%
%

\newpage 

\section{Preconditioning}\label{s:precon}

\subsection{Schur Complement}
The main computational cost in the above algorithm is the solving of the system of equations \xref{pdas1}, \xref{pdas2} and \xref{pdas3}. To do this we firstly introduce some matrix notation. At the $k$-th iteration step we define ${\cal K}^k$, to be $N$ diagonal blocks with the $i$-th block, $ {\cal K}_i^k$, being associated with the active set, ${\cal D}_i^k$, of the $i$-th phase equation of \xref{pdas1}. More precisely,  
$$ 
{{\cal K}_i^k 
:= \tfrac{\varepsilon^2}{\tau} (m_{rj})_{r\in\mathcal{J},j\in \mathcal{D}^k}+  \varepsilon^2( l_{rj})_{ r\in \mathcal{J},j\in\mathcal{D}^k}.}
$$
We further define,
$${\cal B}_{1}^k :=   \left[ \begin{array}{ccccc} -\left( {\bf m}_1^k \right)^T & 0 & 0 & \ldots &  \left( {\bf m}_N^k \right)^T \\
								      0& -\left( {\bf m}_2^k \right)^T & 0 & \ldots & \left( {\bf m}_N^k \right)^T\\
								       & & \ldots & & \\
								       0  & 0 & \ldots & -\left( {\bf m}_{N-1}^k \right)^T & \left( {\bf m}_N^k \right)^T 
\end{array}  \right] \in  \mathbb{R}^{ N-1 \times \omega^k}, $$
where $ \omega^k = \sum_{i=1}^{N} |D_i^k| $
and 
$${\cal B}_{2}^k :=   \left[  -{M}_1^k, -{M}_2^k, \ldots , -{M}_N^k  \right] \in \mathbb{R}^{  |\mathcal{J}|  \times  \omega^k }. $$

At the $k$-th iteration, we may write this linear system in the form; find $\bx^k :=[ \bu^k, \bflambda^k, \Lambda^k]$ such that
$ {K}^k \bx^k = {\bf b}^k,$ 
where the coefficient matrix is of the form
\begin{eqnarray}\label{eq: PDAS linear sys}
{ K}^k := \left[ \begin{array}{ccc} {\cal K}^k & ({\cal B}_{1}^k )^T &( {\cal B}_{2}^k)^T \\
				{\cal B}_{1}^k & 0 & 0 \\
		         	{\cal B}_{2}^k & 0 & 0 	
 \end{array} \right ] .
\end{eqnarray}
For convenience, from here on we will drop the superscript $k$. 
 
To develop a preconditioner for ${ K}$, we write it in the factored form
\begin{eqnarray*}\label{eq: PDAS factored}
{ K} = \left[ \begin{array}{ccc} I & 0 & 0 \\
				{\cal B}_{1} {\cal K}^{-1}& I & 0 \\
		         	{\cal B}_{2} {\cal K}^{-1}& 0 & I 	
 \end{array} \right ] \left[ \begin{array}{ccc} {\cal K} & {\cal B}_{1}^T & {\cal B}_{2}^T \\
				0 & {\cal S}_{11} &{\cal S}_{12} \\
		         	0  & {\cal S}_{21} & {\cal S}_{22} 	
 \end{array} \right ],
\end{eqnarray*}
where ${\cal S}_{ij} := -{\cal B}_{i} {\cal K}^{-1} {\cal B}_{j}^{T}$ for $i,j=1,2$. Rearranging gives,
\begin{eqnarray*}\label{eq: PDAS factored 2}
\left[ \begin{array}{ccc} {\cal K} & {\cal B}_{1}^T & {\cal B}_{2}^T \\
				{\cal B}_{1} & 0 & 0 \\
		         	{\cal B}_{2} & 0 & 0 	
 \end{array} \right ] \left[ \begin{array}{ccc} {\cal K} & {\cal B}_{1}^T & {\cal B}_{2}^T \\
				0 & {\cal S}_{11} &{\cal S}_{12} \\
		         	0  & {\cal S}_{21} & {\cal S}_{22} 	
 \end{array} \right ]^{-1} =  \left[ \begin{array}{ccc} I & 0 & 0 \\
				{\cal B}_{1} {\cal K}^{-1}& I & 0 \\
		         	{\cal B}_{2} {\cal K}^{-1}& 0 & I 	
 \end{array} \right ].
\end{eqnarray*}
Hence, if we choose the preconditioner 
$$  {\cal P}_{\text{exact}} := \left[ \begin{array}{ccc} {\cal K} & {\cal B}_{1}^T & {\cal B}_{2}^T \\
				0 & {\cal S}_{11} &{\cal S}_{12} \\
		         	0  & {\cal S}_{21} & {\cal S}_{22} 	
 \end{array} \right ],$$
the eigenvalues of the preconditioned system have value one, and it can be shown that only two GMRES iterations would be needed in this case, see \cite{murphy2000note}.  

Since each block of ${\cal K}$ consists of the standard finite element matrix for a reaction-diffusion type equation, there exists numerous practical preconditioners and solvers for this block. In particular standard algebraic, or geometric, multigrid can effectively be applied, see \cite{hackbuschbook,stuben2001}. Hence, we are left to find a fully practical approximation to the action of the block matrix 
$$ S_{\text{exact}} := \left[ \begin{array}{cc}
				 {\cal S}_{11} &{\cal S}_{12} \\
		         	 {\cal S}_{21} & {\cal S}_{22} 	
 \end{array} \right ].$$

\subsection{Approximate Schur Preconditioners}

In \cite{blanksarbustoll2012} {the authors consider} a block upper triangular preconditioner of the form:
\begin{equation}
 {P}_{1} := \left[ \begin{array}{ccc} {\cal K} & {\cal B}_{1}^T & {\cal B}_{2}^T \\
				0 & {\tilde{\cal S}}_{11} & 0 \\
		         	0  & 0 & {\tilde{\cal S}}_{22} 	
 \end{array} \right],
  \label{p1}
 \end{equation}
 where the diagonal blocks ${\tilde{\cal S}}_{ii}$, $i=1,2$ are given by
 $$  {\tilde{\cal S}}_{11} := \frac{1}{N_{}} {\cal M} {\cal K}^{-1}  {\cal M} \mbox{ and }   {\tilde{\cal S}}_{22} :=  \left( {\bf 1} \cdot {\bf 1}^T + I \right),$$
{where $I$ is the identity matrix.}  This choice  was shown to lead to the preconditioned system 
$$ K {P}_{1}^{-1} =  \left[ \begin{array}{ccc} I & 0 & 0 \\
				0 & I + E_{11} & E_{12} \\
		         	0  & E_{21} & I + E_{22} 	
 \end{array} \right ], $$
where $E_{ii}$, $i = 1,2$, are such that all eigenvalues are close to zero. 
{Note that with this choice of preconditioner, ${\tilde{\cal S}}_{ij}=0$, $i\neq j$, and hence in \cite{blanksarbustoll2012} the effect of the off diagonal blocks is not considered. }


In this work we propose an alternative approximation by building on the ideas presented in \cite{elman99}. We firstly define,
$ B := \left[  {\cal B}_{1} ,  {\cal B}_{2}\right]^T$, leading to  $ S_{\text{exact}}  = B {\cal K}^{-1} B^T$.  Following  \cite{elman99} we approximate this with 
\begin{eqnarray}\label{BBT} 
S_{\text{exact}}  & \approx & (BB^T) (B {\cal K} B^T)^{-1} (BB^T), \nonumber \\
& = &  \left[ \begin{array}{cc} {\cal B}_{1} {\cal B}_{1}^T  & {\cal B}_{1} {\cal B}_{2}^T \\  
					    {\cal B}_{2} {\cal B}_{1}^T  & {\cal B}_{2} {\cal B}_{2}^T \end{array} \right] (B {\cal K} B^T)^{-1}
					      \left[ \begin{array}{cc} {\cal B}_{1} {\cal B}_{1}^T  & {\cal B}_{1} {\cal B}_{2}^T \\  
					    {\cal B}_{2} {\cal B}_{1}^T  & {\cal B}_{2} {\cal B}_{2}^T \end{array} \right] \nonumber \\
					    & := &  
					    \left[ \begin{array}{cc} D_{11} & D_{12} \\  
					    D_{21} & D_{22}  \end{array} \right] (B {\cal K} B^T)^{-1}
					  \left[ \begin{array}{cc} D_{11} & D_{12} \\  
					    D_{21} & D_{22}  \end{array} \right] := D  (B {\cal K} B^T)^{-1}  D.					    					
\end{eqnarray}
Note that, 
$$D_{11} =   \left[ \begin{array}{ccccc} \alpha_{1} + \alpha_{N}  & \alpha_{N}& \alpha_{N} & \ldots & \alpha_N \\
								\alpha_N & \alpha_2+\alpha_N & \alpha_N & \ldots & \alpha_N \\
								\ldots & \ldots & \ldots & \ldots & \ldots \\
								\alpha_N	& 	\alpha_N	 & \ldots & \ldots & \alpha_{N-1} + \alpha_N						  
\end{array}  \right] \in  \mathbb{R}^{  N-1 \times N-1 } , $$
where $\alpha_i = {\bf m}_i^T {\bf m}_i, i = 1,2,\ldots ,N$
and 
$$D_{22} =  \sum_{i=1}^N M_i M_i^T \in  \mathbb{R}^{ |{\cal J}| \times |{\cal J}|}.$$

\begin{remark}
How well this preconditioner performs is closely related to how well $B$ commutes with ${\cal K}$. We note that the matrix $D_{22}$ is diagonal. Moreover, it is a relatively large proportion of the matrix $B$, and will always be close to a constant diagonal matrix in the regions ${\cal D}^{n,k}_i$. This will even be the case when an adaptive mesh refinement strategy is used to accurately capture the interfacial region, since in this active region the elements are of similar size and shape. Hence, the commutator of $[D_{22}, C]$, with any square matrix $C$ is close to zero. This is a major factor in the quality of the approximation used in \xref{BBT}.
\end{remark}

\vspace{2mm}

We consider two preconditioners developed from the above methodology. Namely, the block diagonal choice
\begin{equation}
 { P}_{2} := \left[ \begin{array}{ccc} {\cal K} & {\cal B}_{1}^T & {\cal B}_{2}^T \\
				0 & D_{11} ({\cal B}_1 {\cal K} {\cal B}_1^T)^{-1} D_{11}   & 0 \\
		         	0  & 0 & D_{22}({\cal B}_2 {\cal K} {\cal B}_2^T)^{-1}   D_{22}
 \end{array} \right ], ~~  
 \label{p2}
 \end{equation}
 and the full approximation
 \begin{equation}
 {P}_{3} := \left[ \begin{array}{cc} {\cal K} &B^T \\
				0 & D  (B {\cal K} B^T)^{-1}  D 
 \end{array} \right].
  \label{p3}
 \end{equation}
 
 \subsection{Practical Preconditioning}
 
The implementation of any of the preconditioners $P_{i}$, $i=1,2,3$, requires a practical and scalable method to calculate the action of their inverses. 
All three preconditioners have inverses that may be written in the form, 
 $$  {P}_{i}^{-1} := \left[ \begin{array}{cc} {\cal K} ^{-1} & 0\\
				0 & I
 \end{array} \right ] \left[ \begin{array}{cc} I & B^T\\
				0 & -I
 \end{array} \right ]\left[ \begin{array}{cc} I & 0 \\
				0 & {{\cal S}}_i^{-1} 
 \end{array} \right ],~~\mbox{for }i=1,2,3$$
where 
$$  {{\cal S}}_1  := \left[ \begin{array}{cc} {\tilde{\cal S}}_{11} & 0 \\
		         	 0 & {\tilde{\cal S}}_{22} 	 
 \end{array} \right ], ~~
 {{\cal S}}_2  := \left[ \begin{array}{cc}  D_{11} ({\cal B}_1 {\cal K} {\cal B}_1^T)^{-1} D_{11}   & 0 \\
		         	 0 & D_{22}({\cal B}_2 {\cal K} {\cal B}_2^T)^{-1}   D_{22}	 
 \end{array} \right ]$$ 
  and 
  $$ 
 {{\cal S}}_3  :=  D  (B {\cal K} B^T)^{-1}  D  .$$
 The main work in calculating the action of these inverses is in calculating the action of the inverses of  ${\cal K}$ and ${\cal S}_i, i=1,2,3$. 
 As mentioned earlier, for the matrix ${\cal K}$ numerous efficient iterative solvers exist. 
 In {the following numerical results section} we chose to apply three Algebraic Multigrid (AMG) V-cycles with simple Gauss-Seidel smoothing, see \cite{stuben2001}. 
 
For $ {{\cal S}}_1$  and ${\cal S}_2$ we are only required to invert a small dense matrix $ {\tilde{\cal S}}_{11}$ and in finding the action of the inverse ${\tilde{\cal S}}_{22}$ we only require the inversion of the diagonal lumped mass matrix. Finally, for  $ {{\cal S}}_3$ we write,
$$
	   D^{-1}
     =    \left[ \begin{array}{cc}  I & 0 \\
		         	 -D_{22}^{-1}D_{21} & I	
 \end{array} \right ]  \left[ \begin{array}{cc} (D_{11} - D_{12} D_{22}^{-1} D_{21})^{-1}  & 0 \\
		         	 0 & D_{22}^{-1}
 \end{array} \right ]  \left[ \begin{array}{cc}  I & -D_{12} D_{22}^{-1}\\
		         	 0 & I	
 \end{array} \right ] .
 $$
We note from earlier remarks that the construction of the matrix $ D_{11} - D_{12} D_{22}^{-1} D_{21} \in  \mathbb{R}^{  N-1 \times N-1 }  $ and its inverse is inexpensive. 

 \subsection{The Two Phase Problem}
 ~~
\begin{theorem}\label{theorem}
When considering a two phase problem, the right preconditioners $P_2$ and $P_3$ are identical.  Moreover, the resulting preconditioned system, $K P_i^{-1}$, $i=2,3$, is a lower triangular matrix with the diagonal consisting of  $1$'s and a solitary  
$$ a =    \frac{{\bf m}_1^T {\cal K}_1^{-1} {\bf m}_1}{ \left( {\bf m}_1^T {\bf m}_1 \right) \left(   {\bf m}_1^T {\cal K}_1 {\bf m}_1 \right)^{-1}\left( {\bf m}_1^T {\bf m}_1 \right) }.$$ 
Furthermore, the minimal polynomial of the resulting system is of degree three and hence GMRES will converge within three iterations. 
\end{theorem}
\begin{proof}
In the two phase case the system to be solved is of the form
$$
{ K} := \left[ \begin{array}{ccc} {\cal K} & ({\cal B}_{1})^T &( {\cal B}_{2})^T \\
				{\cal B}_{1} & 0 & 0 \\
		         	{\cal B}_{2} & 0 & 0 	
 \end{array} \right ] , 
$$
where
$$ {\cal B}_1 = \left[ \begin{array}{cc} - {\bf m}_1^T  &  {\bf m}_1^T  \end{array} \right], ~~
{\cal B}_2 = \left[ \begin{array}{cc} - M_1 &  - M_1 \end{array} \right] \mbox{ and } {\cal K} = 
\left[ \begin{array}{cc} {\cal K}_1 & 0 \\ 0 & {\cal K}_1 \end{array} \right] .$$
This leads to the exact Schur complement,
$$  {\cal S}_{exact} = \left[ \begin{array}{cc} -2 {\bf m}_1^T {\cal K}_1^{-1} {\bf m}_1 & 0 \\ 0 & -2 M_1 {\cal K}_1^{-1} M_1\end{array} \right] .$$
Moreover, since ${\cal B}_1 {\cal B}_2^T = 0$,
$$  {\cal S}_{2} ={\cal S}_{3} = \left[ \begin{array}{cc}  -2 \left( {\bf m}_1^T {\bf m}_1 \right) \left(   {\bf m}_1^T {\cal K}_1 {\bf m}_1 \right)^{-1}\left( {\bf m}_1^T {\bf m}_1 \right)  & 0 \\ 
							     0 & -2 M_1 {\cal K}_1^{-1} M_1   \end{array} \right].  $$
Hence,
$$ {\cal S}_{exact} {\cal S}_{i}^{-1} =   \left[ \begin{array}{cc} a & 0 \\ 0 & I \end{array} \right] , ~~i = 2,3,$$
where $a$ is a scalar. This leads to the full preconditioned system 
$$ K P_2^{-1} = \left[ \begin{array}{cccc} I  &  0 & 0 & 0 \\ 0 & I & 0 & 0 \\  - {\bf m}_1^T {\cal K}_1^{-1} &  {\bf m}_1^T  {\cal K}_1^{-1} & a & 0 \\  - {M}_1^T {\cal K}_1^{-1} &  -{M}_1^T  {\cal K}_1^{-1} & 0 & I  \end{array} \right] .$$
Clearly, this system has only two distinct eigenvalues, $\sigma (  K P_2^{-1} ) = \{ 1, a \}$ and the minimum polynomial for this preconditioned system is of order three. Hence, when using GMRES we would expect to obtain the exact solution in no more than three iterations, see \cite{campbell1996}.
\end{proof}

\section{Numerical Results}\label{s:results}
{In this section we present numerical results that show the efficiency of our proposed preconditioner $P_3$, (\ref{p3}). 
We begin by using exact solves for each matrix in the preconditioning system, then in later results we apply three Algebraic Multigrid (AMG) V-cycles with simple Gauss-Seidel smoothing, see \cite{stuben2001}, for calculating the action of the inverse of ${\cal K}$. 
We denote the approximate preconditioners, in which we use the inexact AMG solver, by $P_{i,AMG}$, $i=1,2,3$.  }

{We note that the fully explicit discretisation of ${\bf A}$ in ${\cal R}_{\bm}^{h,\tau}$ 
leads to an unconditionally well posed discrete problem, allowing the use of large time steps when slow dynamics are encountered, see \cite{KK2}.  
Regarding time stepping, throughout we will use a simple adaptive time stepping strategy whereby: }
\begin{enumerate}
\item If the number of PDAS steps required to obtain ${\bf u}^{n+1}$ are fewer than $5$, we set $\tau_{n+1} = 1.1 \tau_n$. 
\item If the number of PDAS steps required to obtain ${\bf u}^{n+1}$  is between $5$ and $10$, $\tau_n$ remains unchanged.
\item If the number of PDAS steps required to obtain ${\bf u}^{n+1}$  exceeds $10$, we recalculate ${\bf u}^{n+1}$ with a time step reduced by a half, $\tau_{n} = 0.5 \tau_n$. 
\end{enumerate}
At time step $n$ with initial time step $\tau_{n}$ and previous solution ${\bf u}^n$ using the PDAS iteration scheme calculate ${\bf u}^{n+1}$. This solution is taken when the active set size does not change between iterations. 
We found that this led to a practical stable method.
In two space dimensions we set $\Omega = [0,1]^2$ and in three space dimensions we set $\Omega = [-0.5,0.5]^3$. \\

\begin{remark}\label{tolerance remark}
During initial calculations, it was seen that the proposed  PDAS scheme was only reliable when a high tolerance was enforced on  linear solve. Hence, throughout the following results we will apply a tolerance on the relative GMRES error of $1e$--$10$. Given this constraint on the PDAS scheme, it is critical that a robust and efficient solver is used.
\end{remark}

\subsection{Two Space Dimensions}

\subsubsection{ Grain Coarsening}\label{sd}

The first problem we consider is that of
grain coarsening in which we start with an initially well mixed mixture of $N$ phases. 
{The mixture rapidly separates into bulk regions of each phase, with typically each phase having multiple bulk regions}. Once this fast dynamical process has taken place, the bulk regions then slowly diffuse, see Figure \ref{fig: evolution} where the motion of eight phases from $T=1$ to $T= 100$ is presented.

\begin{figure}[h]
  \centering
  \subfloat[$T=1$]{\label{fig:spin8_a_2D}\includegraphics[width=0.33\textwidth]{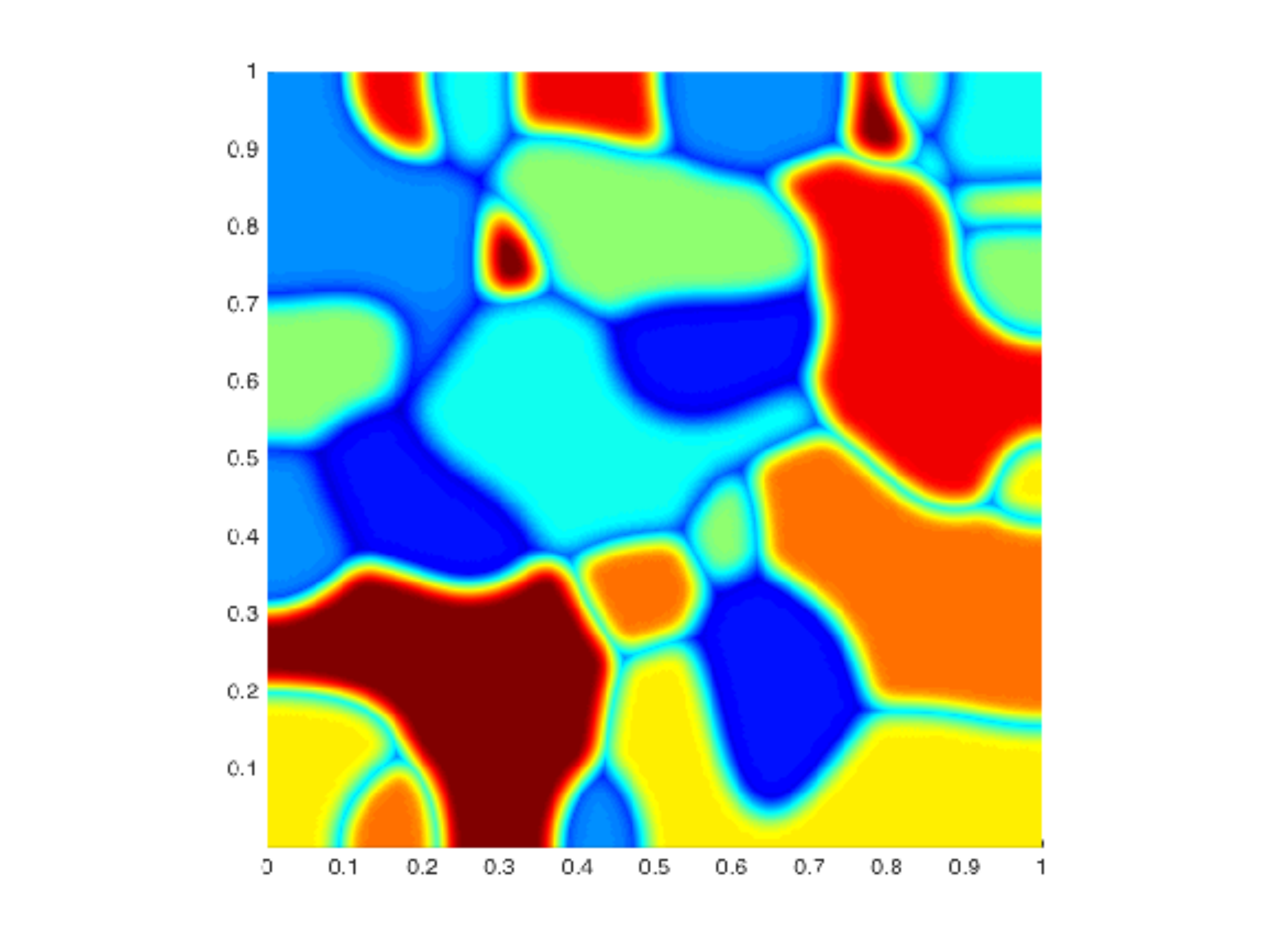}}
  \subfloat[$T=10$]{\label{fig:spin8_b_2D}\includegraphics[width=0.33\textwidth]{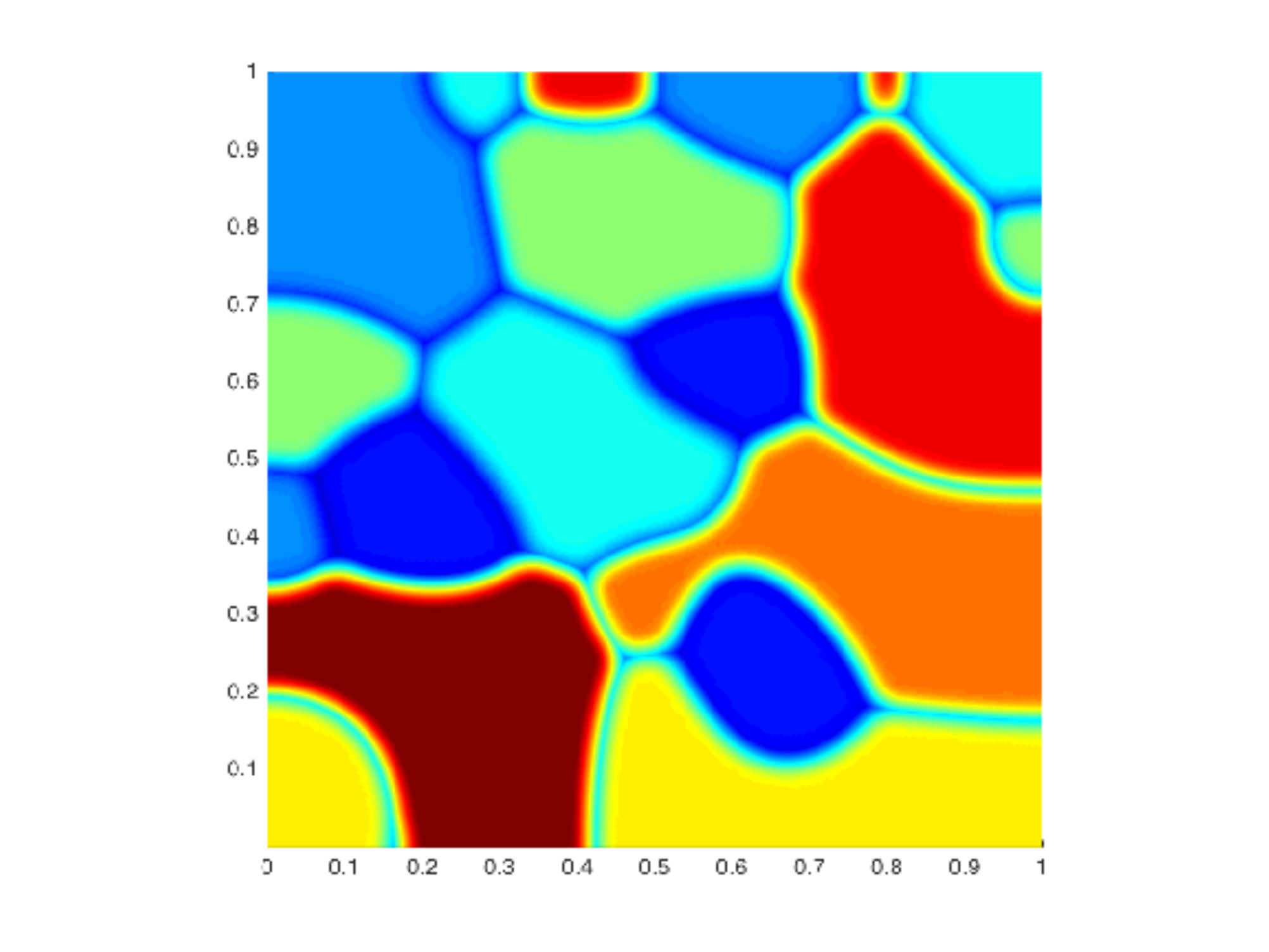}}
  \subfloat[$T=100$]{\label{fig:spin8_c_2D}\includegraphics[width=0.33\textwidth]{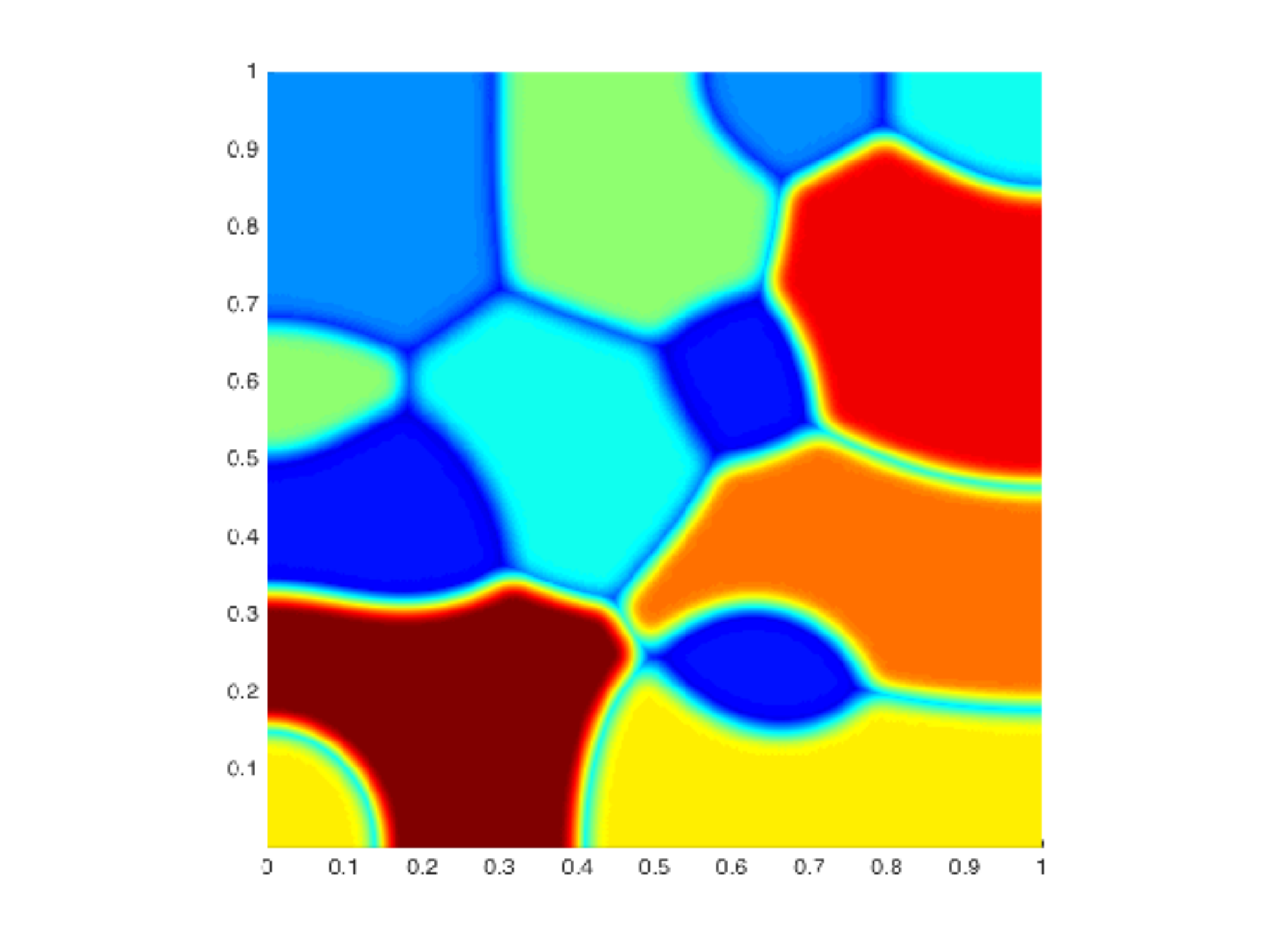}}
\caption{Evolution of eight initially well mixed phases.}
  \label{fig: evolution}
\end{figure}

\begin{table}
\begin{center}
\begin{tabular}{ | c || c | c | c || c | c | c || c | c | c |  }
\hline
\multicolumn{10}{|c|}{Two Phases} \\
\hline
	 &\multicolumn{3}{|c||}{Mesh 1 ($|\mathcal{J}|=3765$)} &  \multicolumn{3}{|c||}{Mesh 2 ($|\mathcal{J}|=14889$)} &  \multicolumn{3}{|c|}{Mesh 3 ($|\mathcal{J}|=59217$)} \\ 
	 \hline
	$\varepsilon$  &  $P_1$ &  $P_2$ &  $P_3$  &  $P_1$ &  $P_2$ &  $P_3$  &  $P_1$ &  $P_2$ &  $P_3$ \\
	 \hline
	0.04  &$ 35/{\bf 24}$ & $3$  & $3$ & $ 29/{\bf 21} $& $3$    & $3$  &  $25/{\bf 20}$  & $3$ & $3$   \\ 
	\hline
	0.02  & $36/{\bf 24}$   & $3$  & $3$  & $31/{\bf 25} $   & $3$  & $3$  &  $27/{\bf 24}$   & $3$  & $3$  \\ 
	\hline
	0.01  & $36/{\bf 22}$  & $3$  & $3$ & $31/{\bf 26} $   & $3$  & $3$ &  $28/{\bf 26}$   & $3$ & $3$   \\ 
	\hline
\hline
	\multicolumn{10}{|c|}{Four Phases} \\
	 \hline
	$\varepsilon$  &  $P_1$ &  $P_2$ &  $P_3$  &  $P_1$ &  $P_2$ &  $P_3$  &  $P_1$ &  $P_2$ &  $P_3$ \\
	 \hline
	0.04  & $70/{\bf 49}$ & $17/{\bf 12}$ & $10/{\bf 8}$ &   $64/{\bf 42} $  &   $ 18/{\bf 12}$ &    $ 13/{\bf 10}$&    $62/{\bf 40} $  &  $17/{\bf 13} $ &  $ 11/{\bf 10}$  \\ 
	\hline
	0.02  &    $ 64/{\bf 47}$   &    $ 15/{\bf 13}$    &    $9/{\bf 8}$  &   $ 64/{\bf 50}$   &   $16/{\bf 13} $   &  $10/{\bf 9} $  &  $ 62/{\bf 45} $  &  $18/{\bf 12}  $ &  $ 12/{\bf 11}$\\ 
	\hline
	0.01 & $60/{\bf 41}$   & $14/{\bf 13}$ & $8/{\bf 7}$ &$58/{\bf 49}  $ &$ 15/{\bf 14}$ &$ 10/{\bf 8}$& $61/{\bf 47} $    &$ 18/{\bf 13}  $& $11/{\bf 9} $  \\ 
	\hline \hline
\multicolumn{10}{|c|}{Six Phases} \\
	 \hline
	  $\varepsilon$  &  $P_1$ &  $P_2$ &  $P_3$  &  $P_1$ &  $P_2$ &  $P_3$  &  $P_1$ &  $P_2$ &  $P_3$ \\
	 \hline
	0.04   & $92/{\bf 57}$  &  $ 22/{\bf 15}$ & $12/{\bf 9}$ &  $  84/{\bf 63}$   &  $25/{\bf 16} $ & $15/{\bf 10} $ &  $ 89/{\bf 49}$    &  $ 24/{\bf 18}$ &  $ 11/{\bf 10}$  \\ 
	\hline
	0.02   &  $88/{\bf 65} $ & $ 21/{\bf 17} $& $9/{\bf 8} $ &     $ 85/{\bf 65}$&   $ 21/{\bf 17} $&  $ 13/{\bf 11}$&  $95/{\bf 66} $  &  $23/{\bf 17} $& $ 12/{\bf 10}$ \\ 
	\hline
	0.01  &   $ 75/{\bf 55}$ &  $ 18/{\bf 17}$ & $  9/{\bf 7} $ & $ 75/{\bf 65}$    &  $  21/{\bf 18} $ &   $ 12/{\bf 10}$&    $ 91/{\bf 61}$  &   $22/{\bf 17} $&  $12/{\bf 10}$  \\ 
	\hline
\end{tabular}
\caption{Maximum GMRES iteration counts when starting with a well mixed initial condition.}
 \label{tab: spin 2D} 
\end{center}
\end{table}

We compare the performance of the three preconditioners $P_1$, (\ref{p1}), $P_2$, (\ref{p2}), and $P_3$, (\ref{p3}), with respect to: the number of phases, the interface width parameter $\varepsilon$, and the mesh size. In each case we use exact solves for each matrix in the preconditioning system. {The initial mesh, Mesh 1, has mesh size $h \approx 1/32$ and all other meshes are uniform refinements of this mesh, the number of nodes of each mesh is given by $|\mathcal{J}|$. In Table \ref{tab: spin 2D} we display the maximum number of GMRES iteration counts together with the average number, in the form $a/{\bf b}$, where $a$ is the maximum number and ${\bf b}$ is the average number. We consider three meshes, three values of $\varepsilon$, and three values of $N$. 
From this table, for each of the three meshes and each of the three values of $\varepsilon$, we see the dependence of $P_1$ on the number of phases, $N$. A similar dependence can be seen for the choice $P_2$, albeit a milder one. It is $P_3$ that outperforms the other two choices in this regard, as it shows almost no dependence on phase number for each of the three meshes and each of the three values of $\varepsilon$.  In addition, for the two phase problem, $N=2$, we observe the three iteration convergence, stated in Theorem \ref{theorem}.

\begin{figure}[h]
  \centering
  \subfloat[$P_1$ with $\varepsilon = 0.01$, two phases and Mesh 4.]{\includegraphics[width=0.45\textwidth]{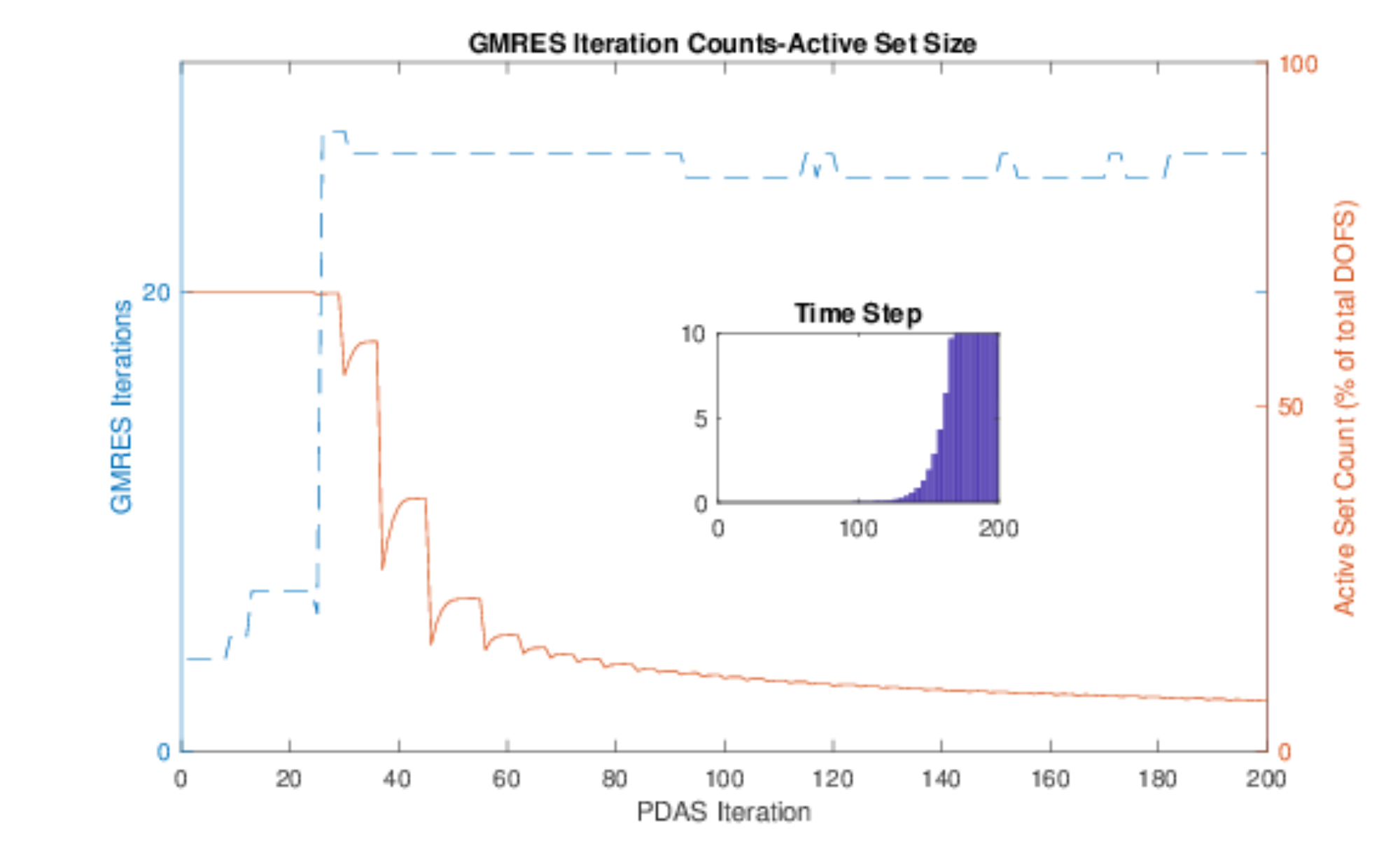}}~~~
  \subfloat[$P_3$ with $\varepsilon = 0.01$, two phases and Mesh 4.]{\includegraphics[width=0.45\textwidth]{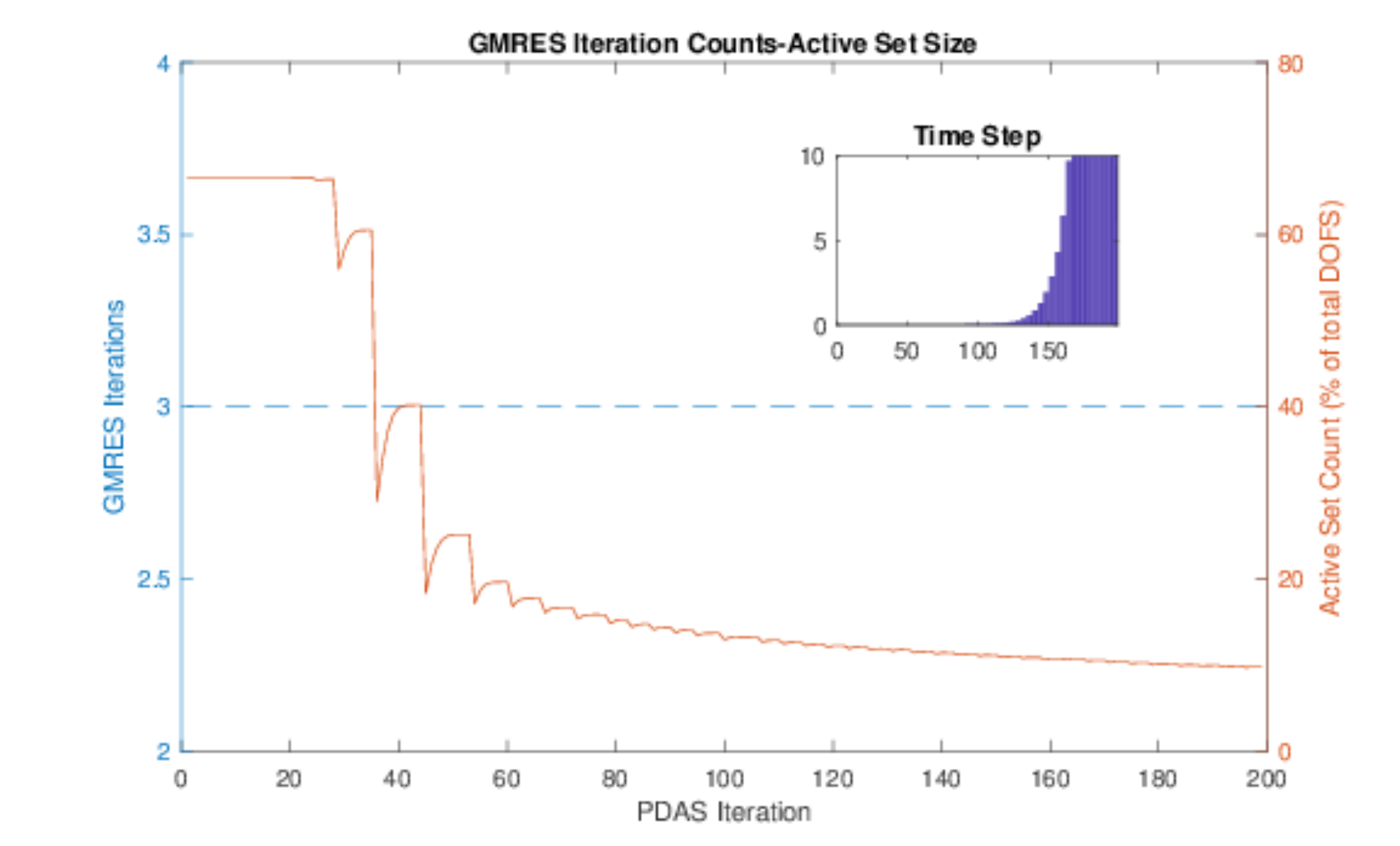}} \\
 \subfloat[$P_1$ with $\varepsilon = 0.01$, four phases and Mesh 4.]{\includegraphics[width=0.45\textwidth]{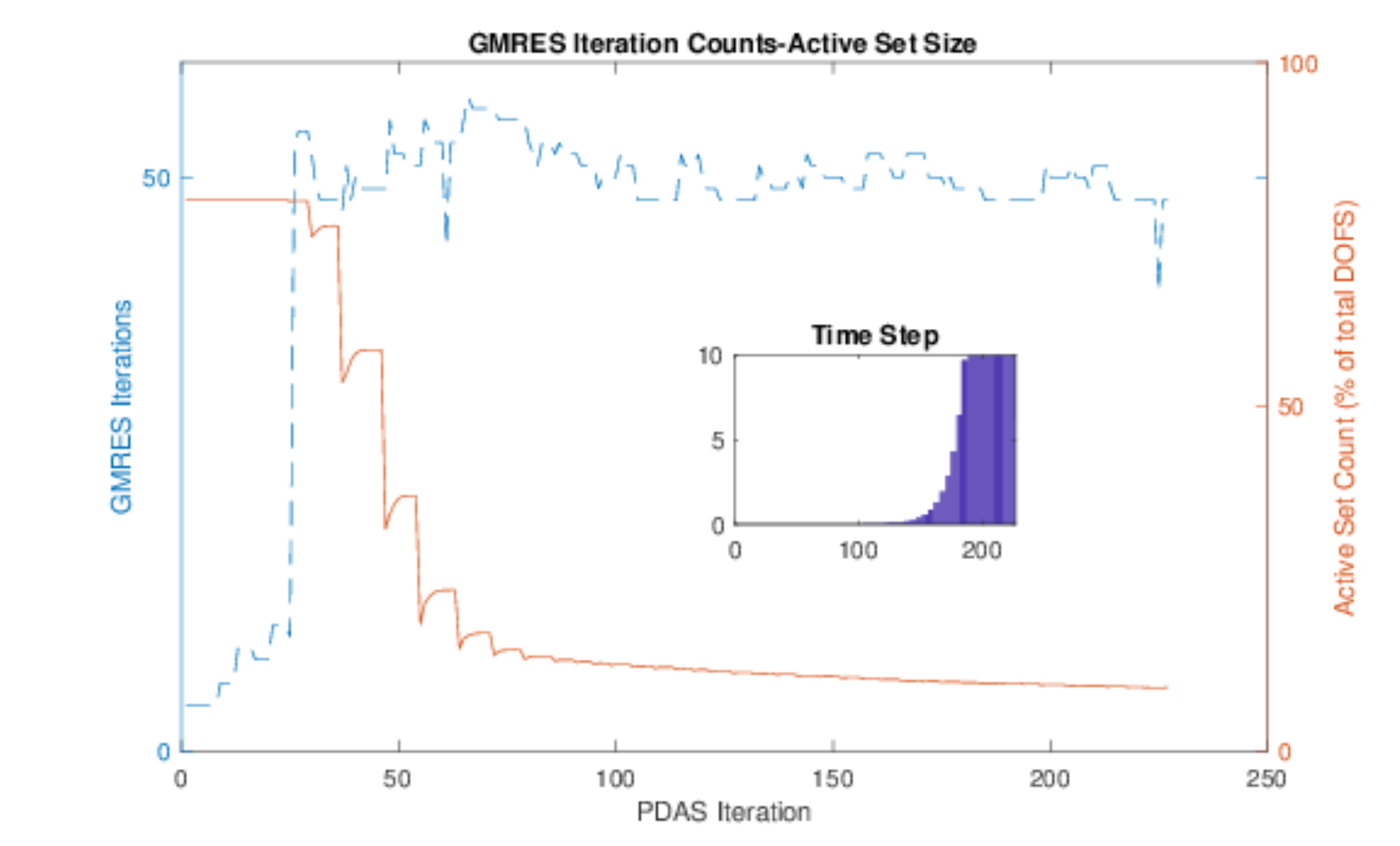}}~~~
  \subfloat[$P_3$ with $\varepsilon = 0.01$, four phases and Mesh 4.]{\includegraphics[width=0.45\textwidth]{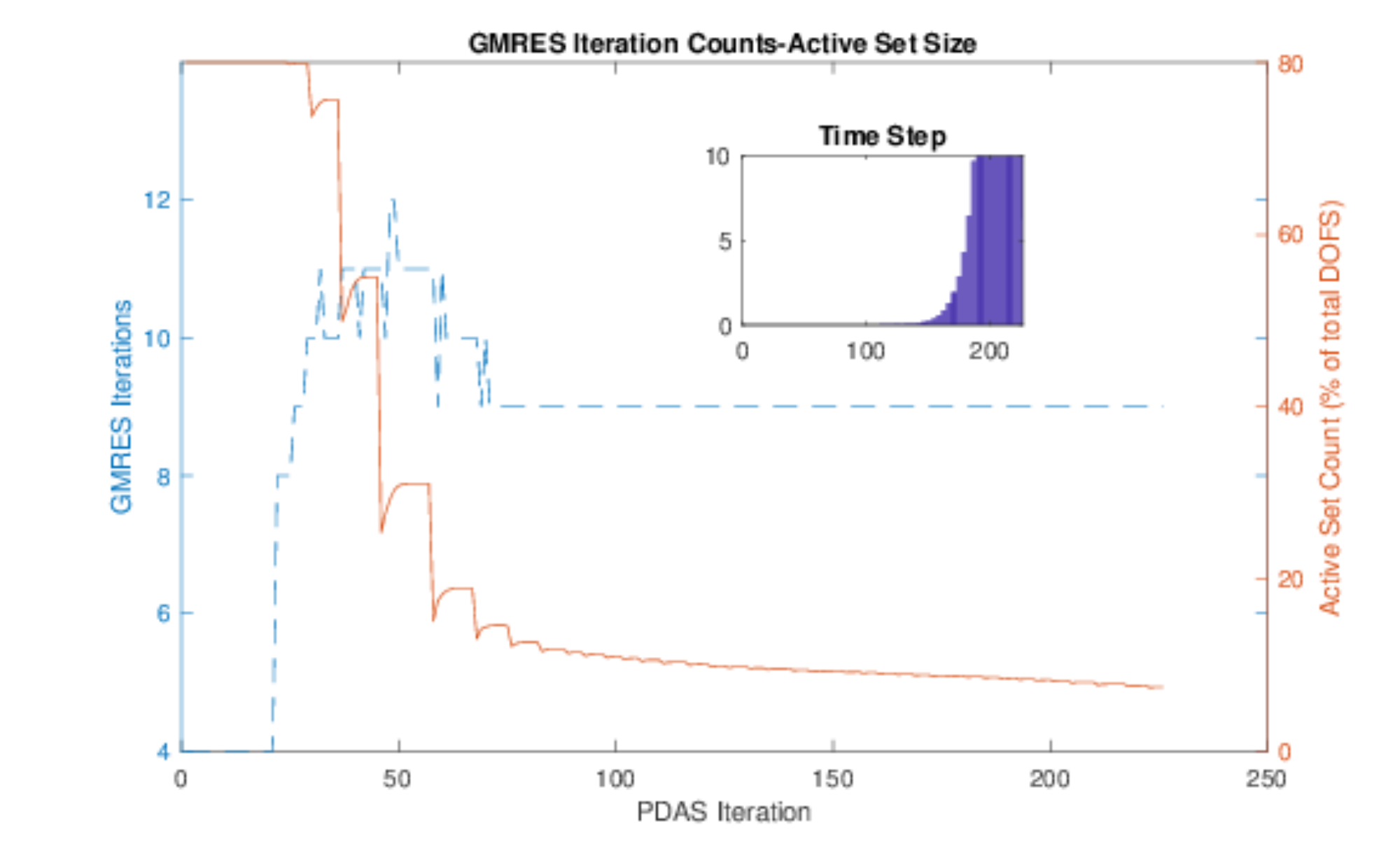}} \\
 \subfloat[$P_1$ with $\varepsilon = 0.01$, six phases and Mesh 4.]{\includegraphics[width=0.45\textwidth]{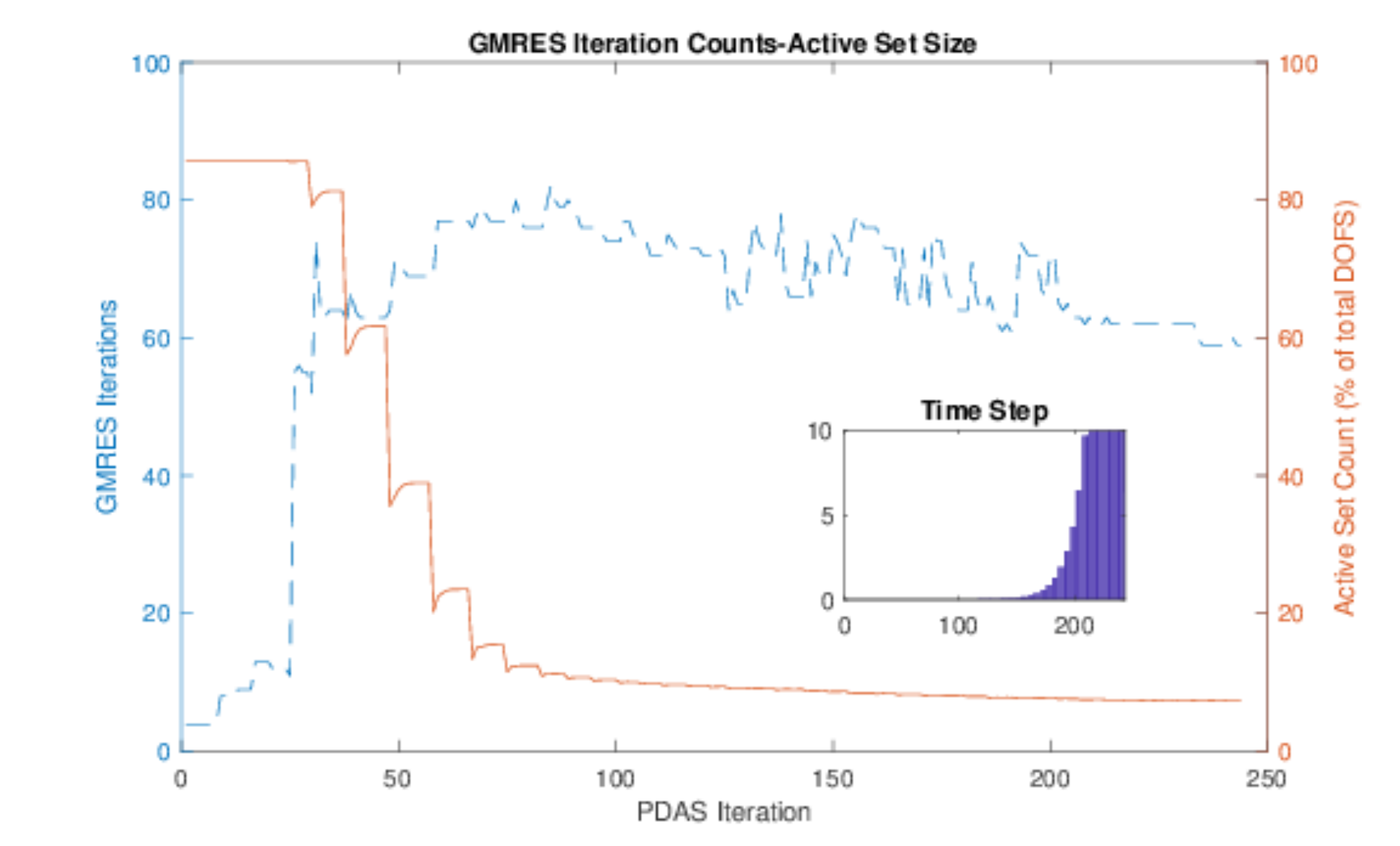}}~~~
  \subfloat[$P_3$ with $\varepsilon = 0.01$, six phases and Mesh 4.]{\includegraphics[width=0.45\textwidth]{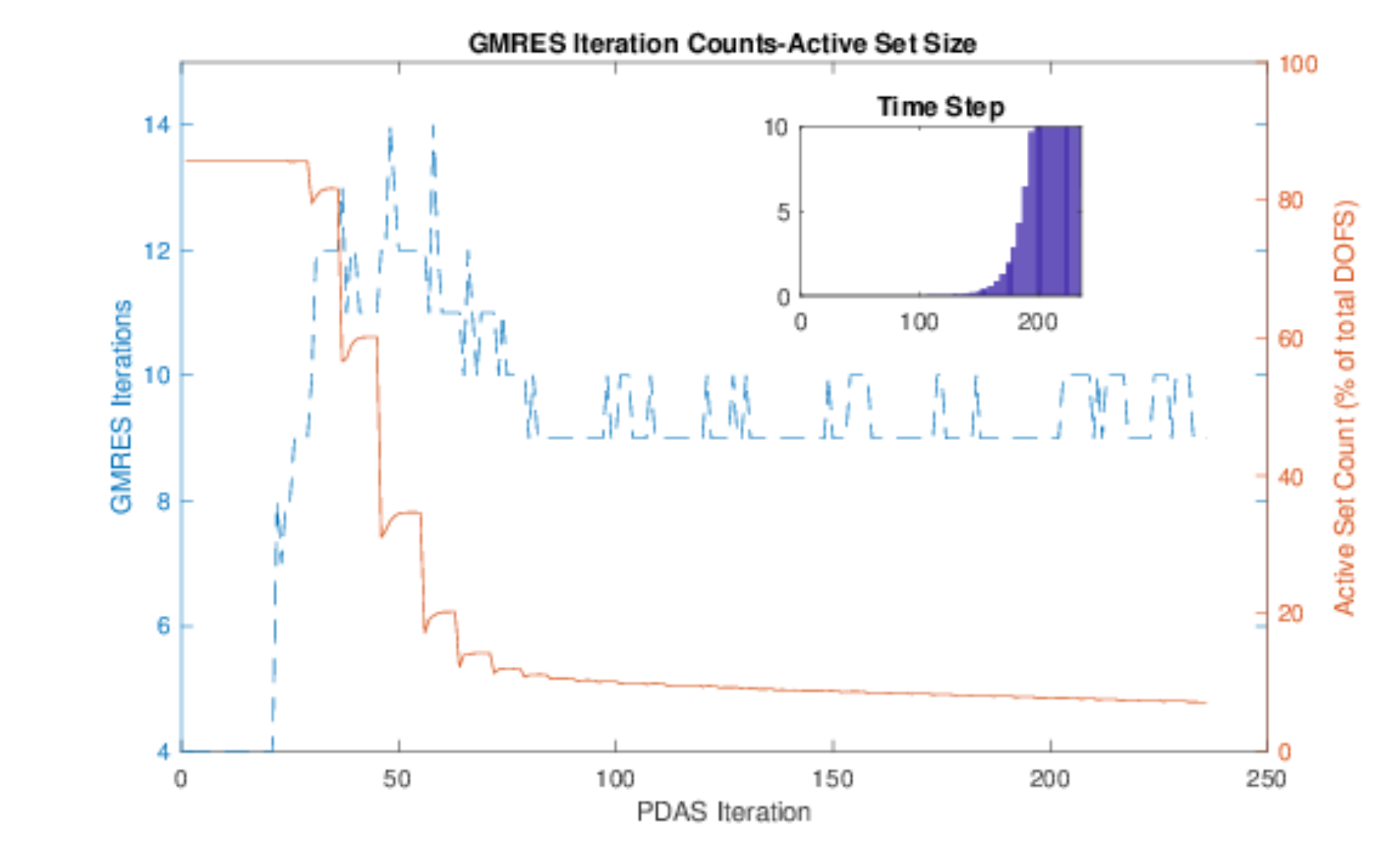}}
  \caption{Phase dependence of preconditioners $P_1$ (left) and $P_3$ (right).}
  \label{fig: Stoll-Kay phase}
\end{figure}
\begin{figure}[h]
  \centering
 \subfloat[$P_1$ with $\varepsilon = 0.02$, four phases and Mesh 3]
 {\includegraphics[width=0.45\textwidth]{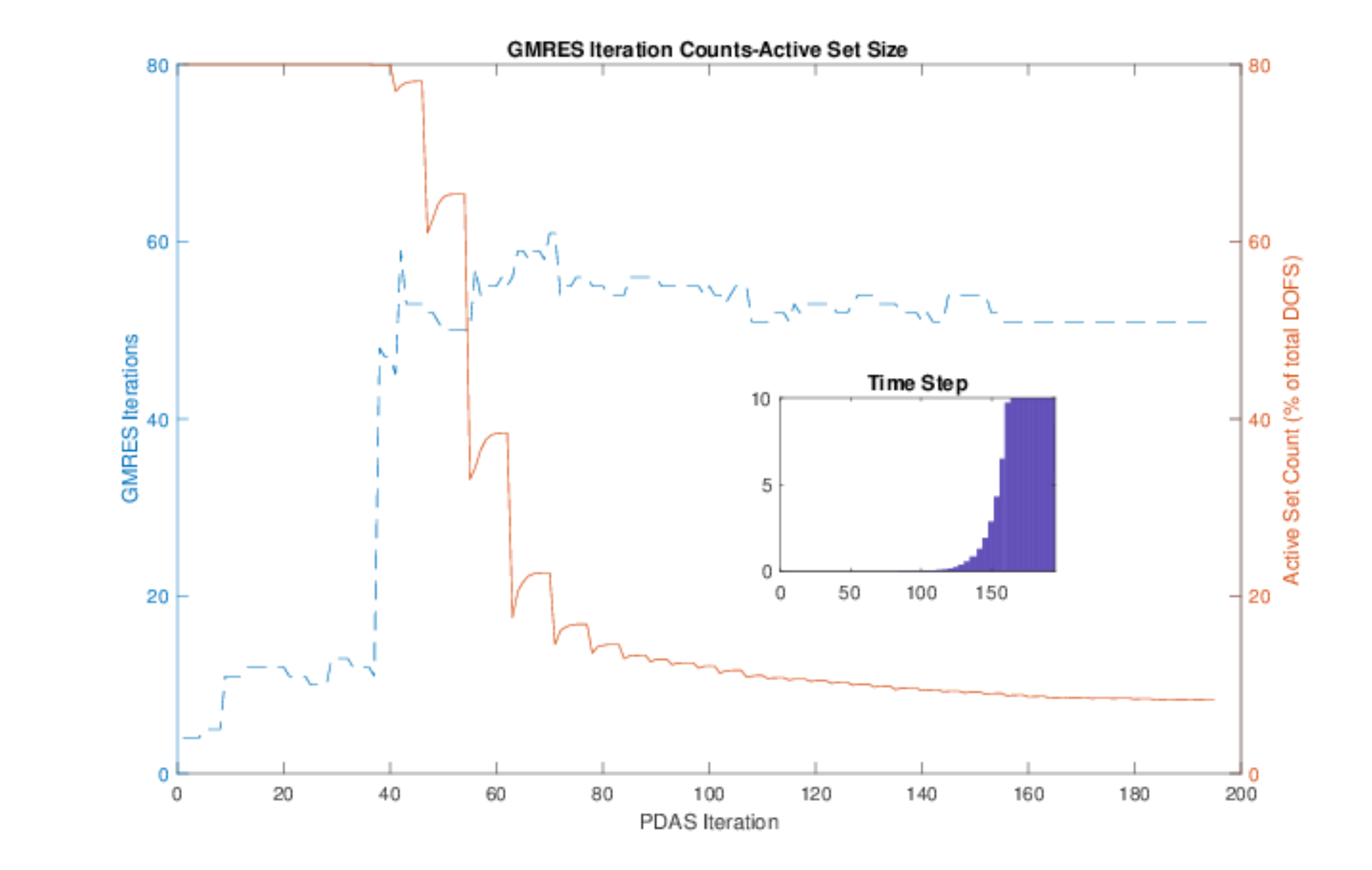}}~~~
  \subfloat[$P_3$ with $\varepsilon = 0.02$, four phases and Mesh 3]
  {\includegraphics[width=0.45\textwidth]{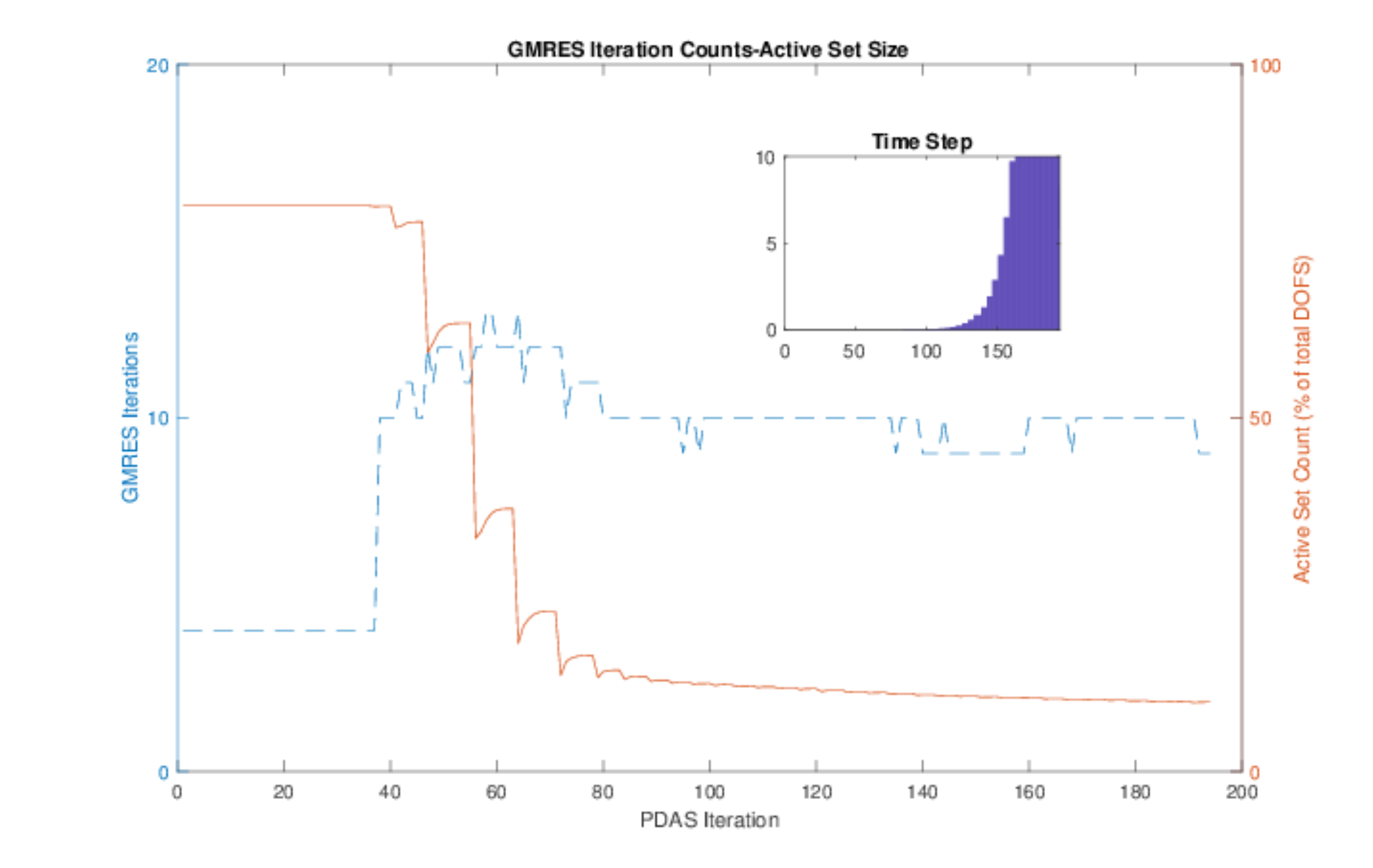}} \\
  \subfloat[$P_1$ with $\varepsilon = 0.02$, four phases and Mesh 4]
 {\includegraphics[width=0.45\textwidth]{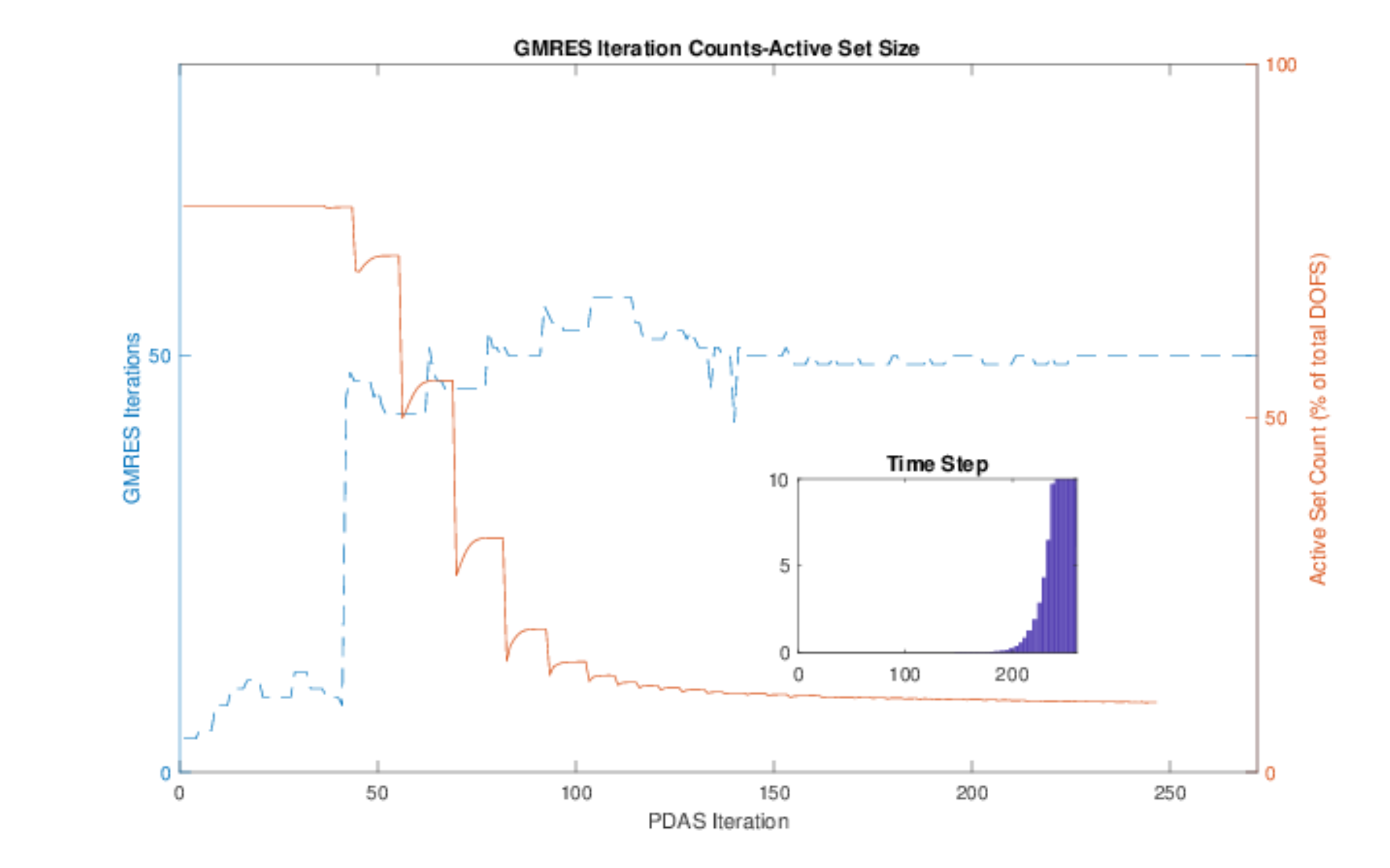}}~~~
  \subfloat[$P_3$ with $\varepsilon = 0.02$, four phases and Mesh 4]
  {\includegraphics[width=0.45\textwidth]{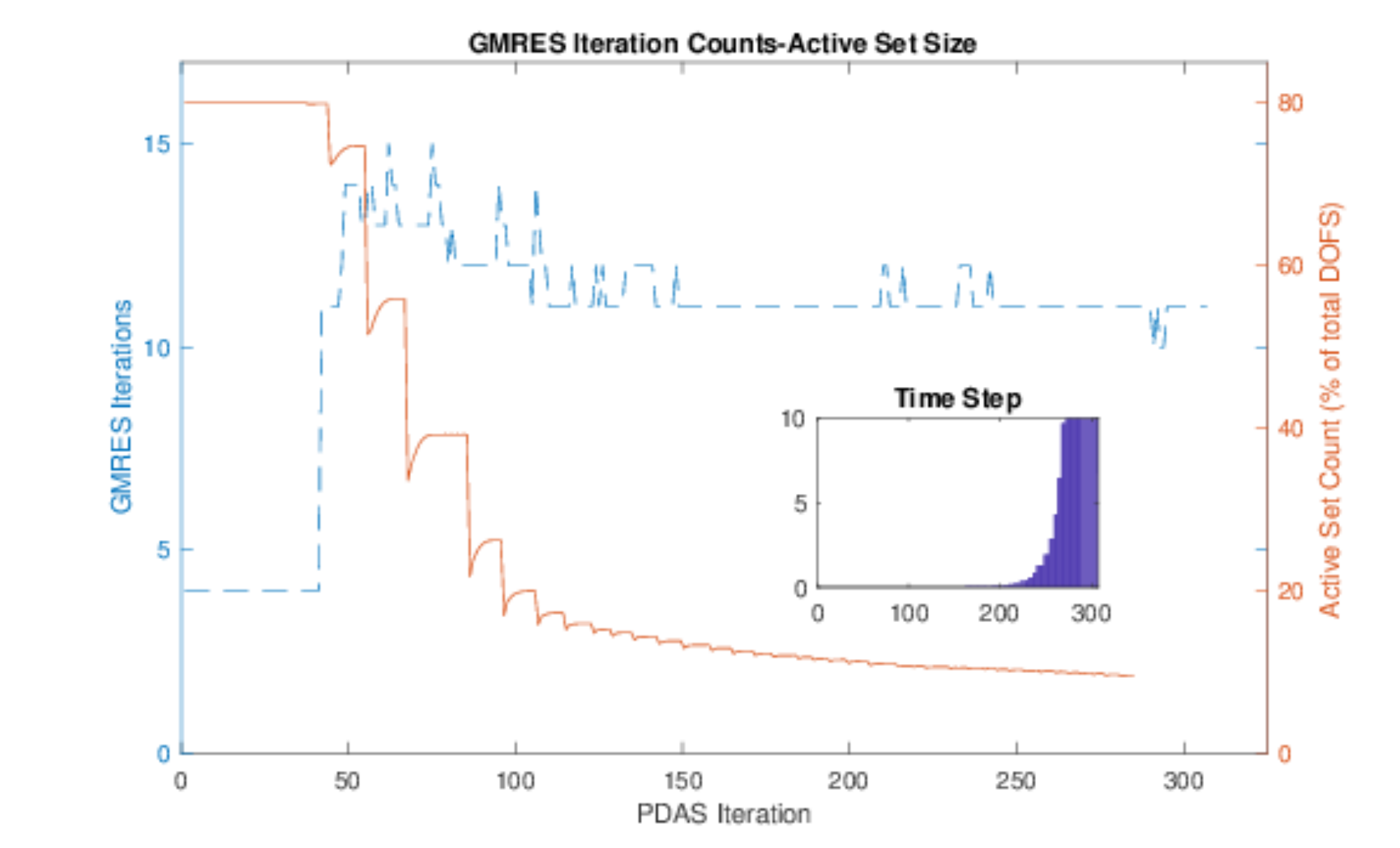}}
  \caption{Mesh dependence}
  \label{fig: Stoll-Kay mesh}
\end{figure}

\begin{figure}[h]
  \centering
 \subfloat[$P_1$ with $\varepsilon = 0.02$, four phases and Mesh 4]
 {\includegraphics[width=0.45\textwidth]{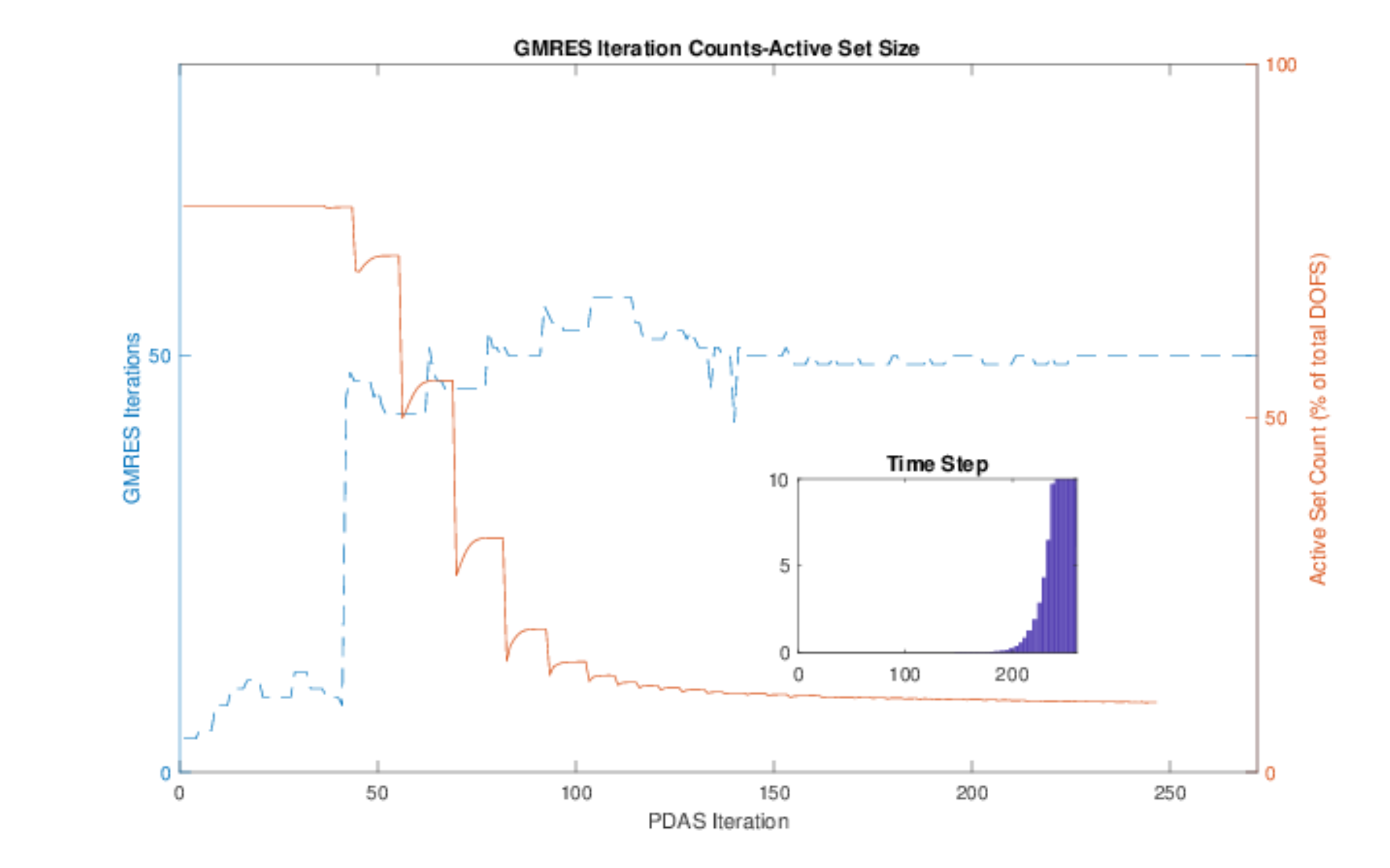}}~~~
  \subfloat[$P_3$ with $\varepsilon = 0.02$, four phases and Mesh 4]
  {\includegraphics[width=0.45\textwidth]{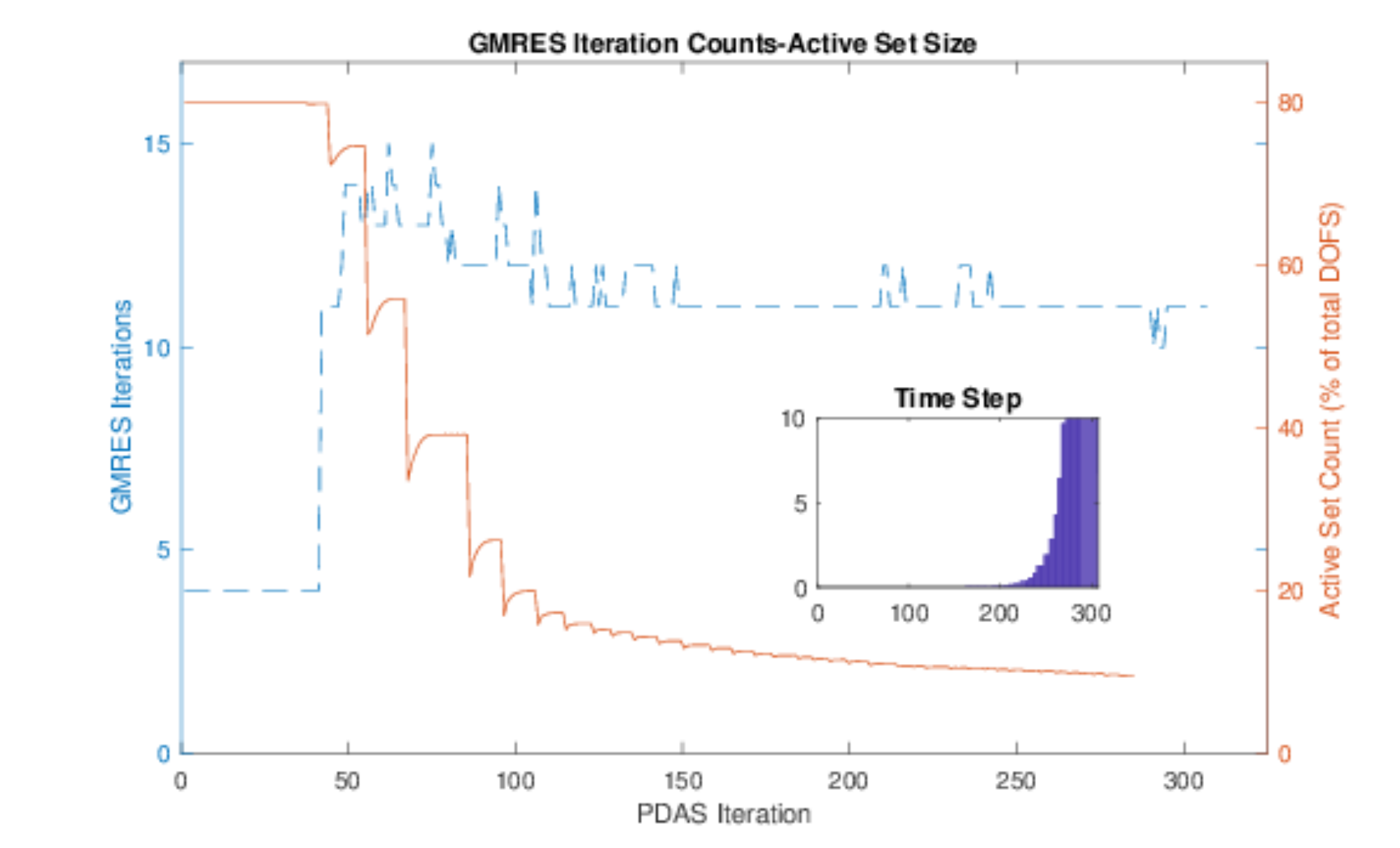}} \\
  \subfloat[$P_1$ with $\varepsilon = 0.01$, four phases and Mesh 4]
 {\includegraphics[width=0.45\textwidth]{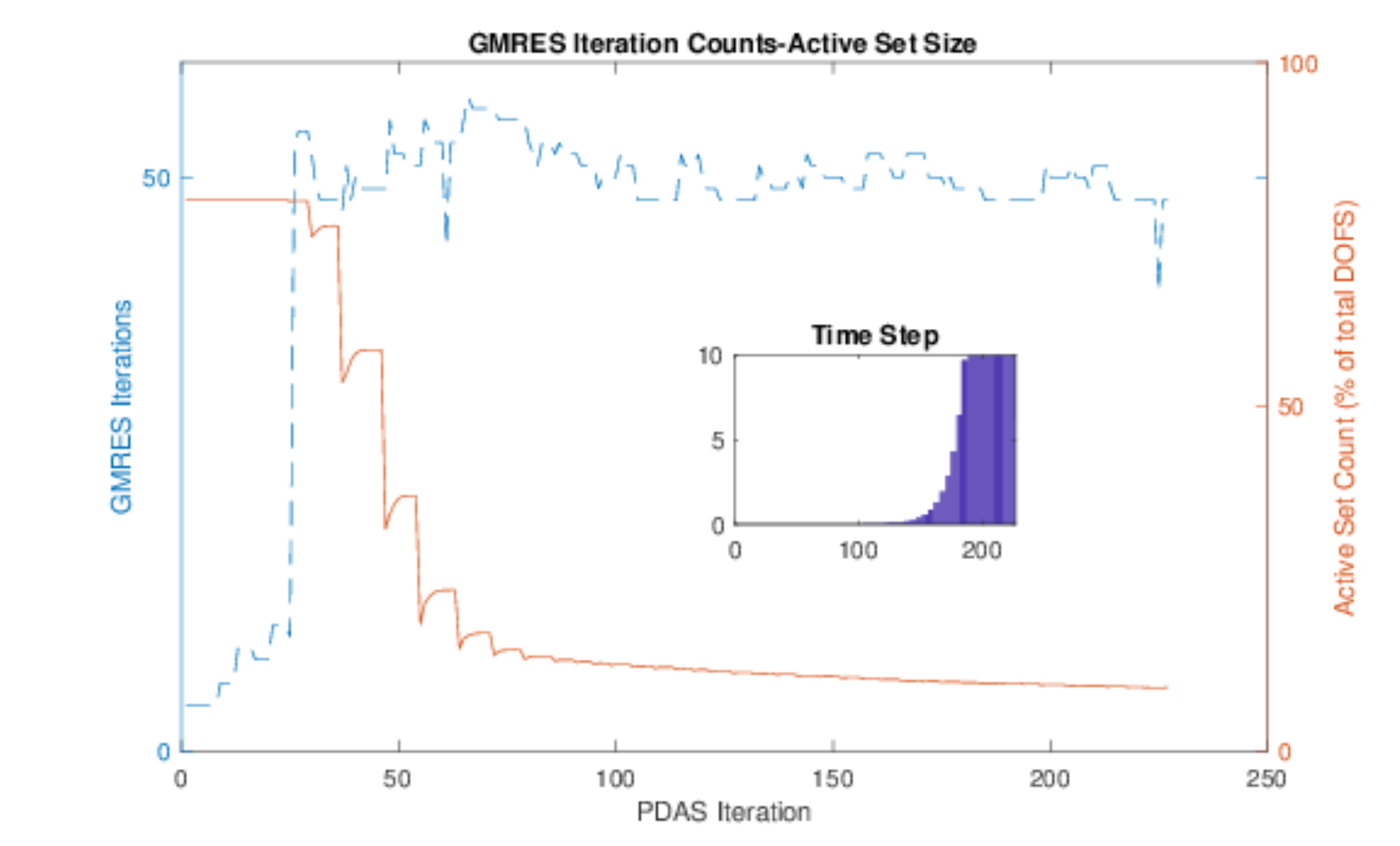}}~~~
  \subfloat[$P_3$ with $\varepsilon = 0.01$, four phases and Mesh 4]
  {\includegraphics[width=0.45\textwidth]{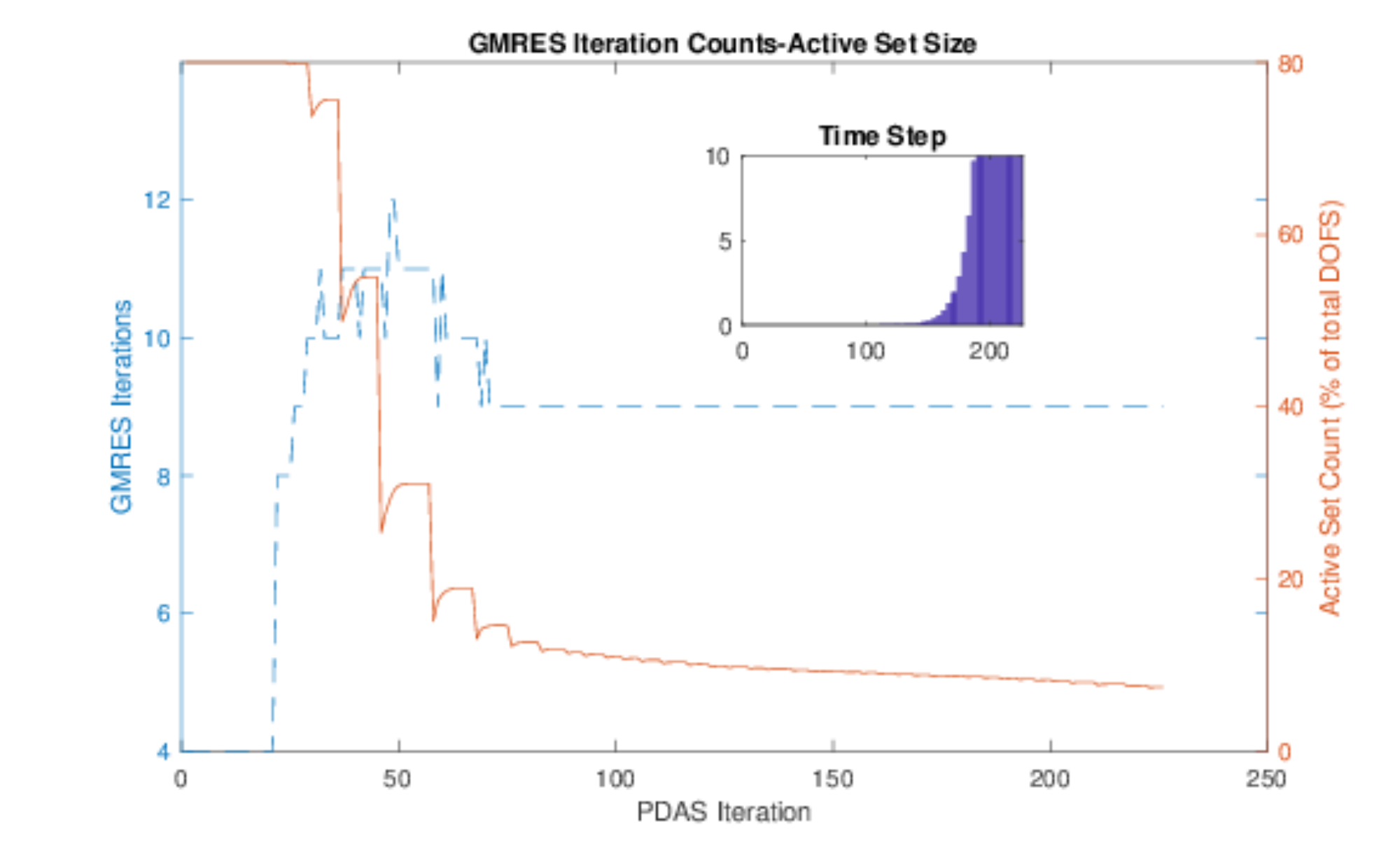}}
  \caption{$\varepsilon$ dependence}
  \label{fig: Stoll-Kay epsilon}
\end{figure}


In Figure \ref{fig: Stoll-Kay phase} we further investigate how the number of phases, $N$, affects the performance of preconditioners $P_1$ and $P_3$. We ignore $P_2$ since $P_3$ is computationally similar whilst having superior convergence rates.  We display the number of GMRES iterations throughout a simulation together with the percentage of the total number of DOFs that the active sets make up. In addition we show the effect of the active time stepping by displaying the time step size throughout the simulation. 
We set $\varepsilon=0.01$ and use Mesh 4, which has $|\mathcal{J}| =236868$. 
Similar to Table \ref{tab: spin 2D}, we see a strong dependence for $P_1$ but a very weak dependence for $P_3$. 
The results in Figure  \ref{fig: Stoll-Kay mesh} are displayed in the same format as those in Figure \ref{fig: Stoll-Kay phase}, but here we investigate the effect that the mesh size has on the performance of preconditioners $P_1$ and $P_3$. In particular we set $\varepsilon=0.02$ and $N=4$ and we show results for Mesh 3, for which $|\mathcal{J}|=59217$, and  Mesh 4. We conclude with Figure \ref{fig: Stoll-Kay epsilon} in which we set $N=4$ and use Mesh 4, with $\varepsilon=0.02$ and $\varepsilon=0.01$ to see the effect that $\varepsilon$ has on the performance of preconditioners $P_1$ and $P_3$. {In both Figures  \ref{fig: Stoll-Kay mesh} and \ref{fig: Stoll-Kay epsilon} we again see a strong dependence for $P_1$ but a very weak dependence for $P_3$.}

\subsubsection{Quadruple Junction to Triple Junction}

We now turn to an initial condition of fully formed bulk regions. We set {$\varepsilon=0.005$ and $N=5$. For the initial data we consider a square, consisting of four phases of bulk square regions, that is surrounded by a fifth phase. This unstable initial geometry rapidly evolves so that the quadruple junction is replaced by two triple junctions with $120^\circ$ angles, see Figure \ref{fig: 5 phase bulk}. }
\begin{figure}[h]
  \centering
  \subfloat[$T=0$]{\label{fig:Bulk5_b_2D}\includegraphics[width=0.33\textwidth]{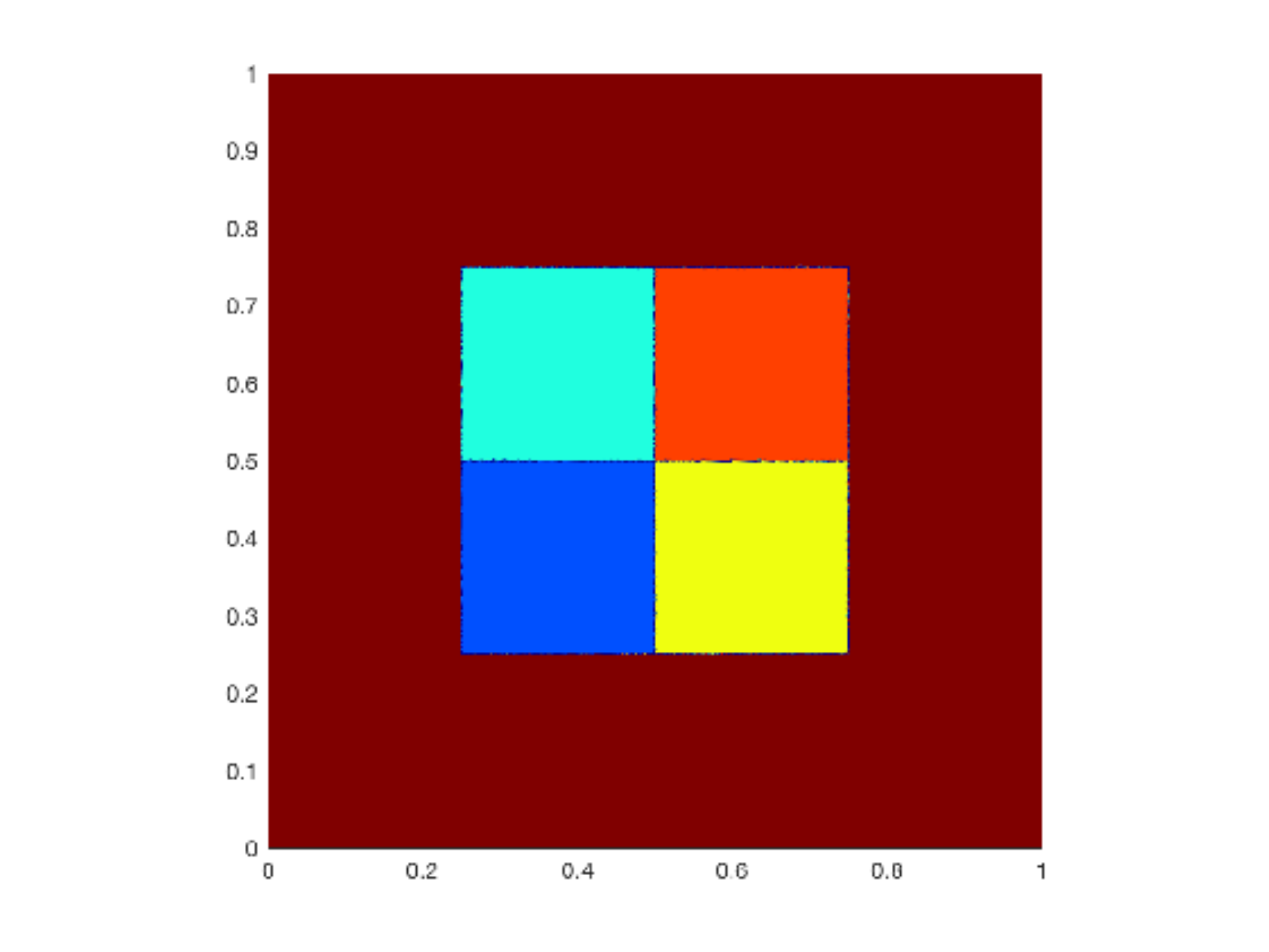}} 
  \subfloat[$T=1$]{\label{fig:Bulk5_c_2D}\includegraphics[width=0.33\textwidth]{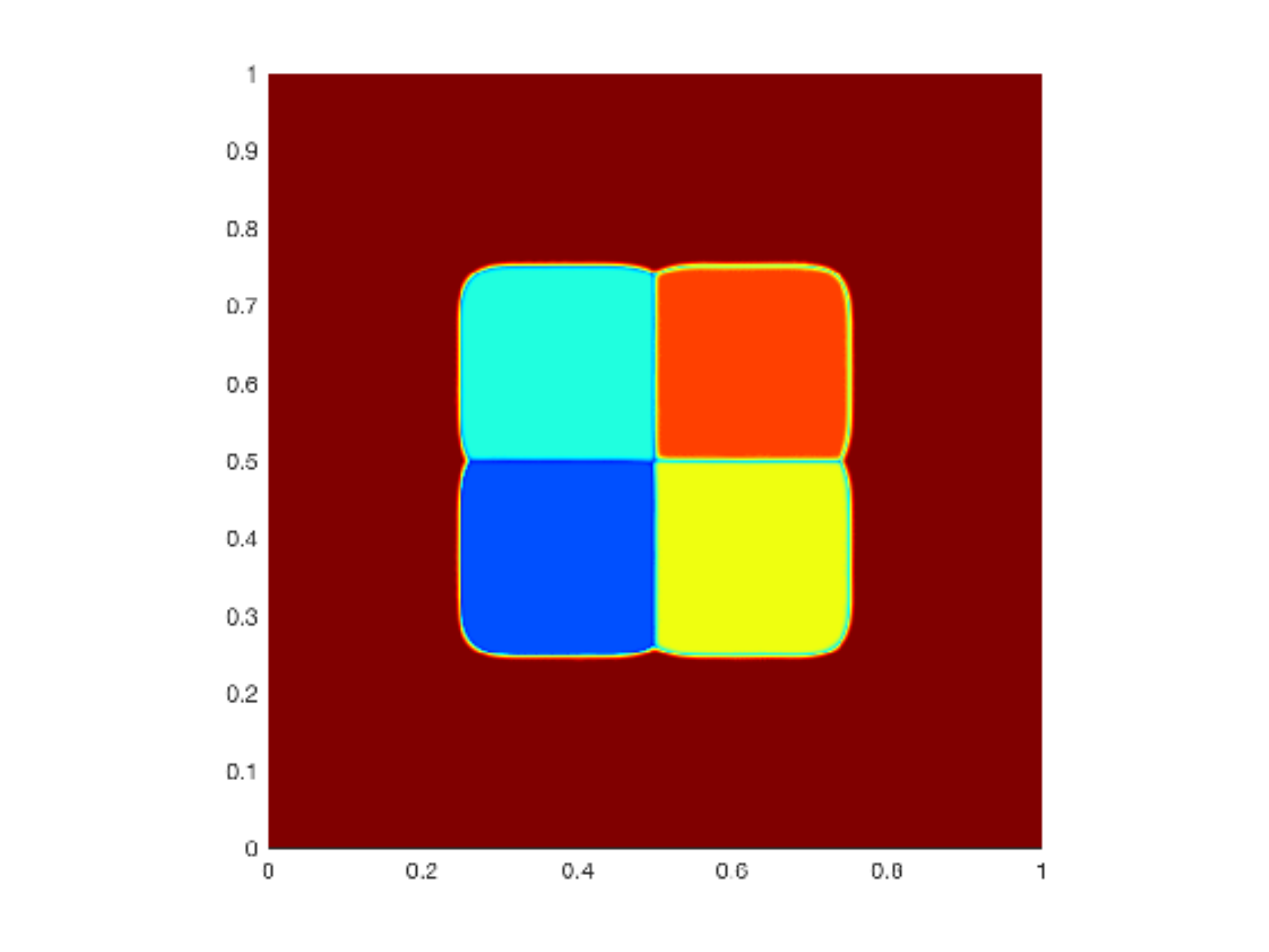}}
  \subfloat[$T=5000$]{\label{fig:Bulk5_d_2D}\includegraphics[width=0.33\textwidth]{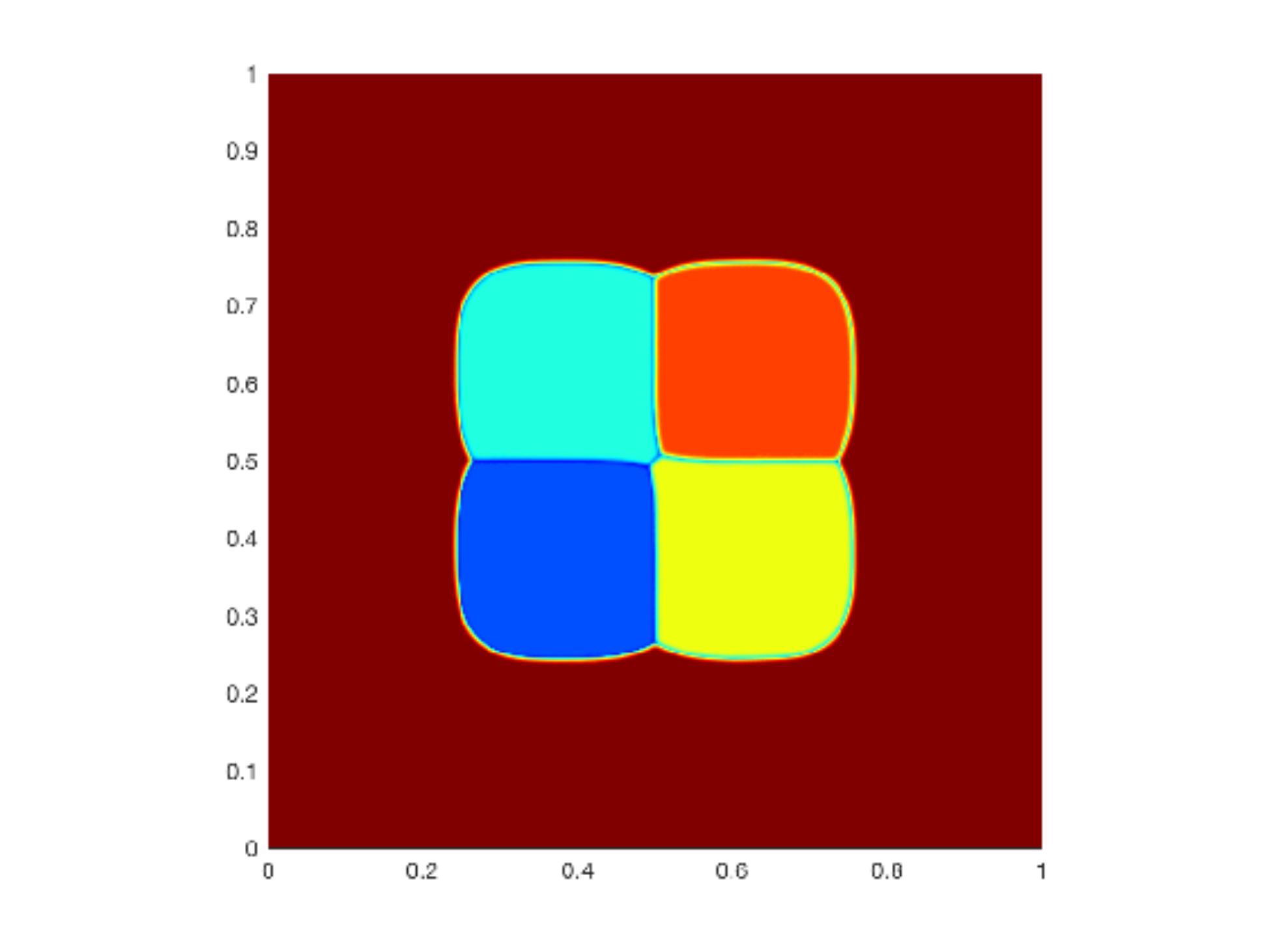}} \\
  \subfloat[$T=0$]{\label{fig:Bulk5_e_2D}\includegraphics[width=0.33\textwidth]{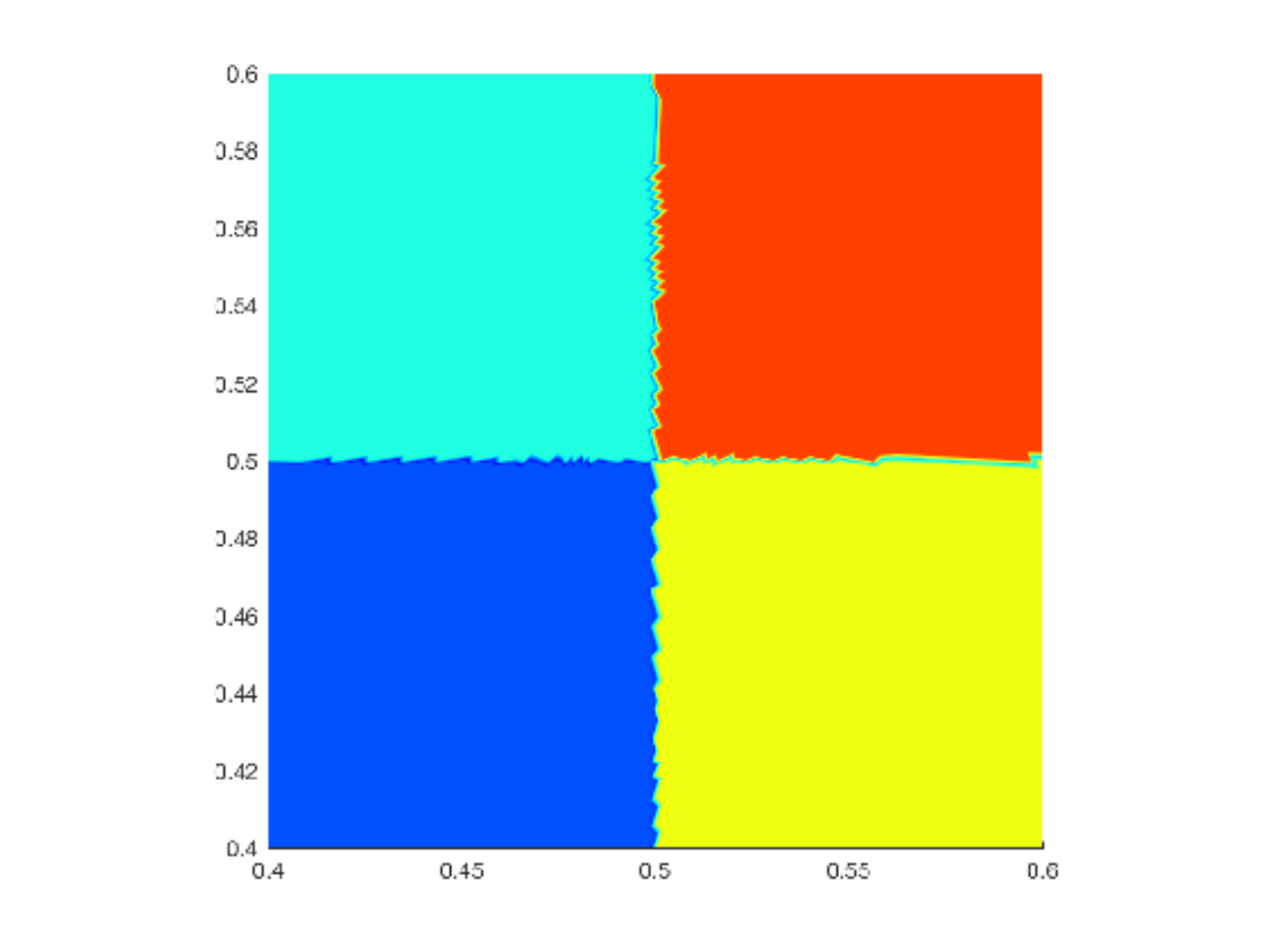}} 
  \subfloat[$T=1$]{\label{fig:Bulk5_f_2D}\includegraphics[width=0.33\textwidth]{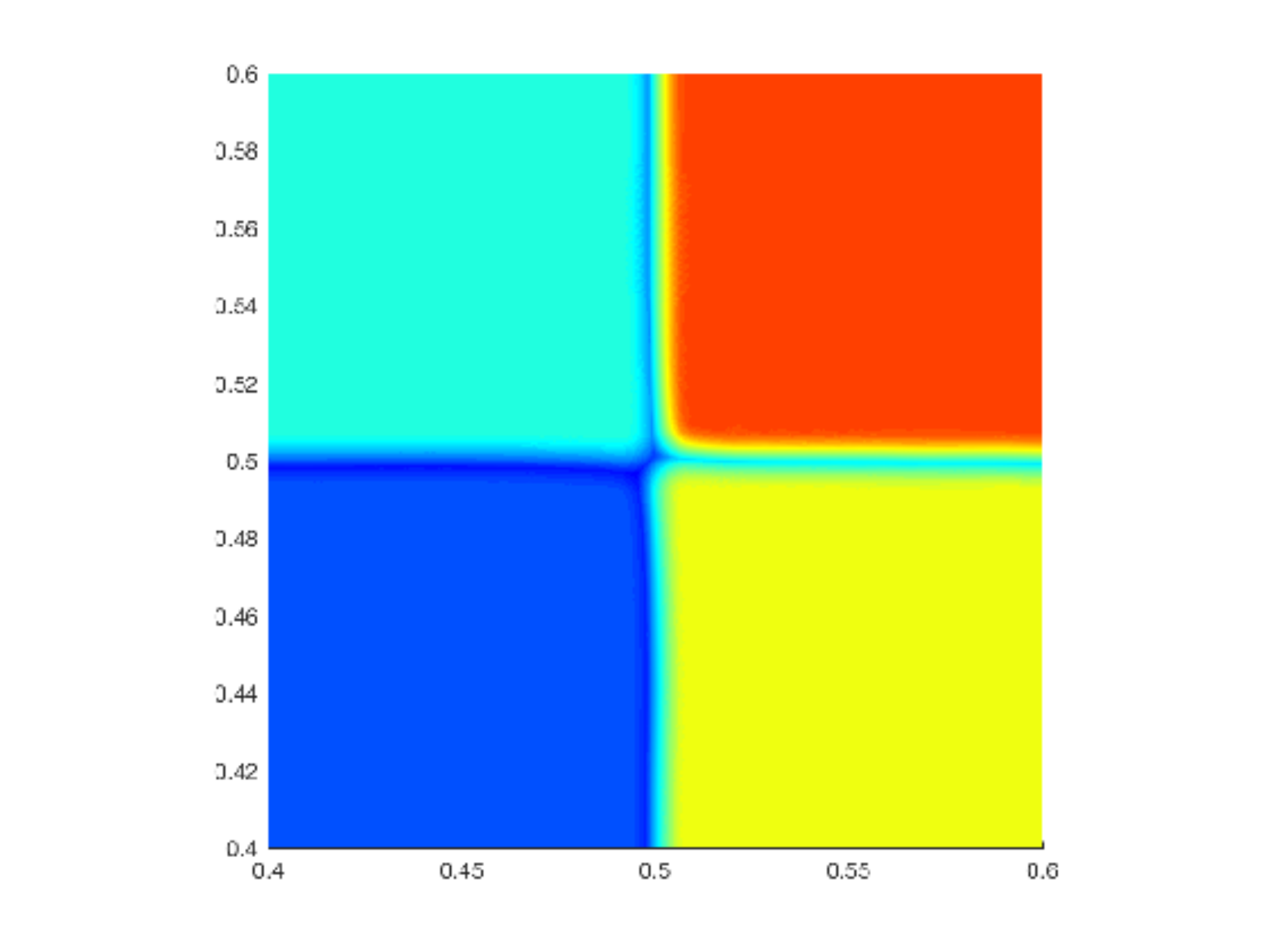}}
  \subfloat[$T=5000$]{\label{fig:Bulk5_g_2D}\includegraphics[width=0.33\textwidth]{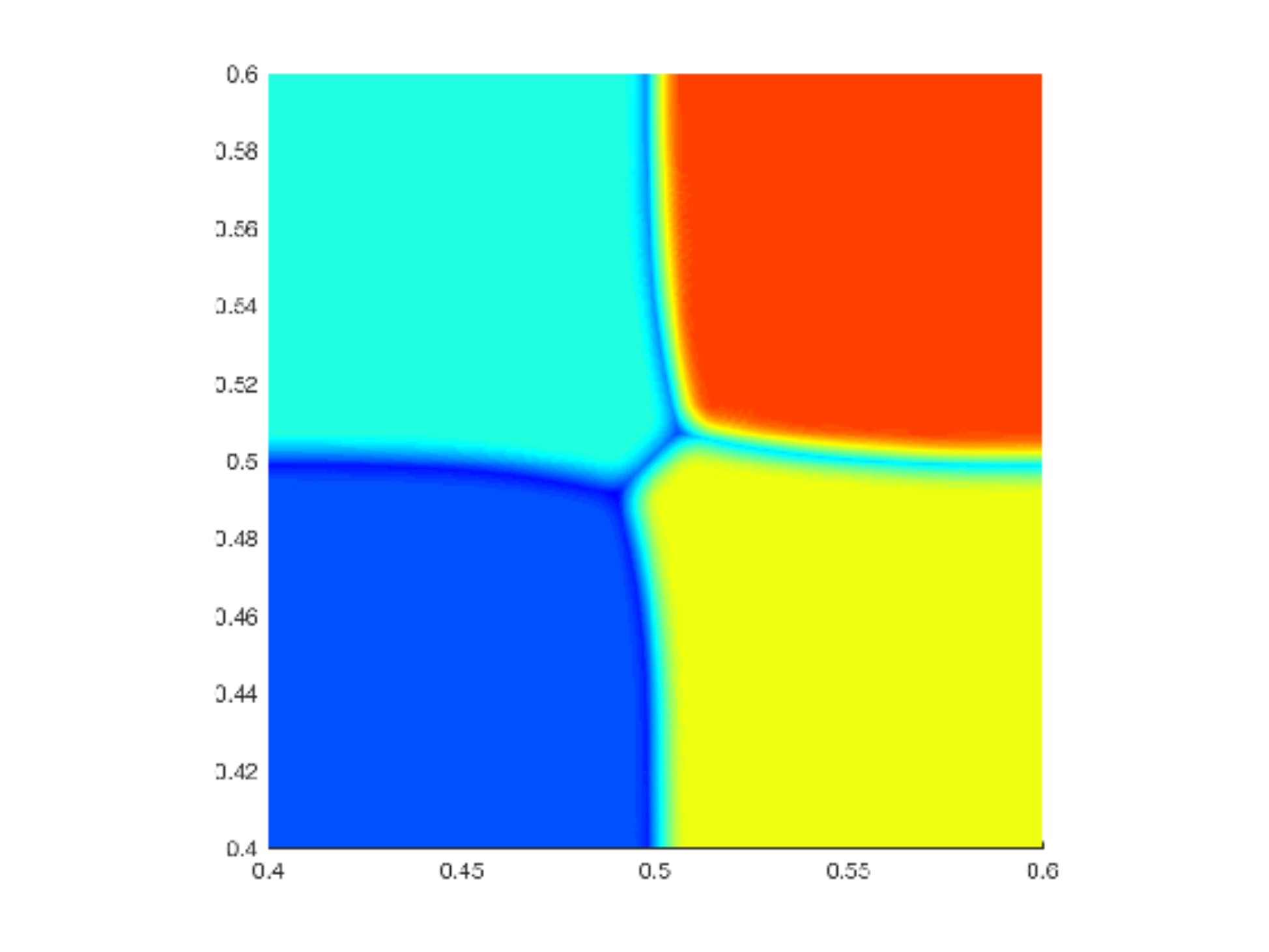}} \\
 \caption{Time evolution of four bulk phases surrounded by a fifth phase, $\varepsilon = 1/200$ and $236868$ DOFs for each phase.}
  \label{fig: 5 phase bulk}
\end{figure}
\begin{figure}[h]
  \centering
  \subfloat[GMRES Iterations and active set size]{\label{fig:Bulk5_a_2D}\includegraphics[width=0.5\textwidth]{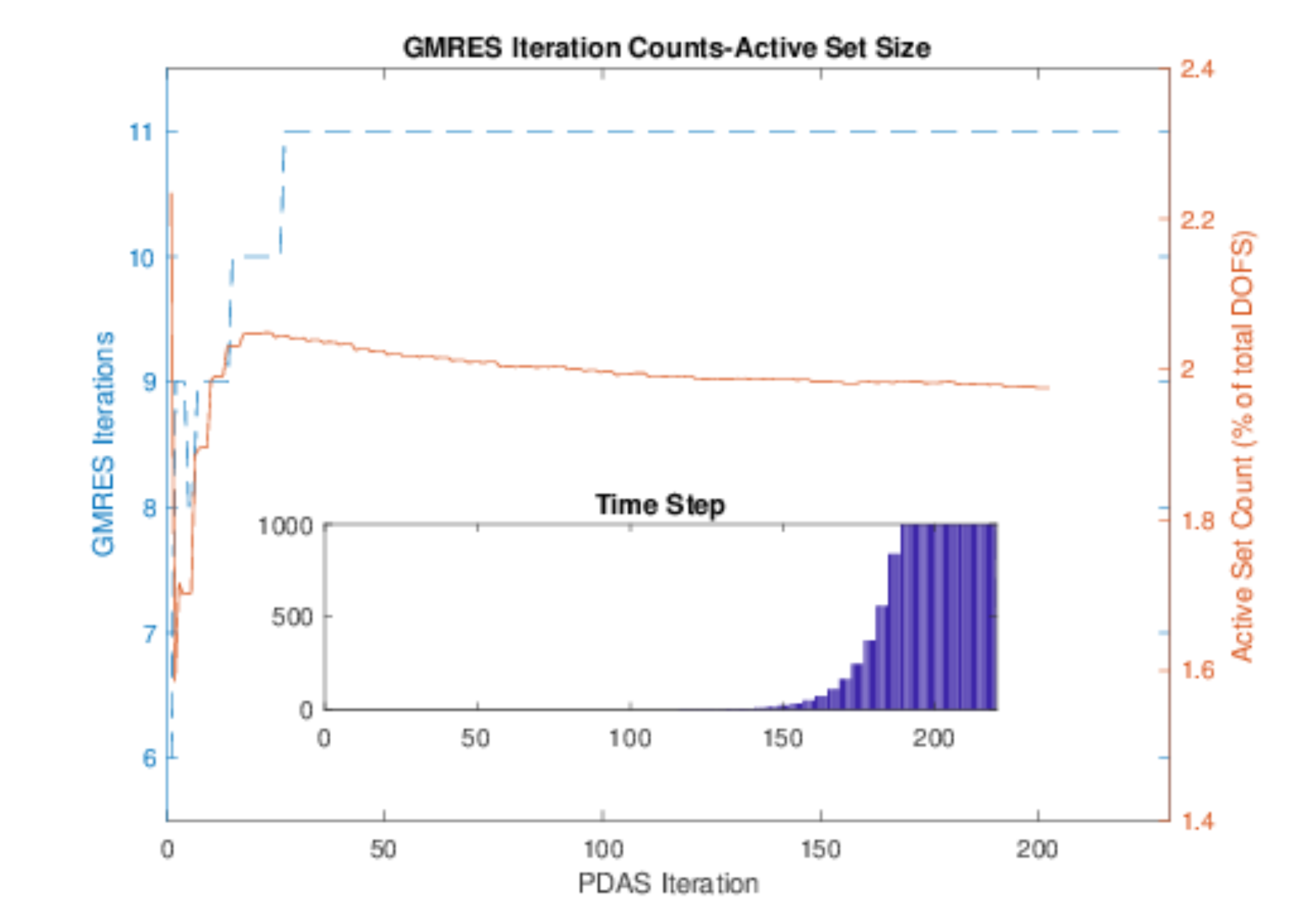}} 
 \caption{GMRES iteration count using $P_{3,AMG}$ and active set size for quadruple junction, $\varepsilon = 1/200$ and $236868$ DOFs for each phase.}
  \label{fig: 5 bulk GMRES}
\end{figure}

In Figure \ref{fig: 5 bulk GMRES} we present the iteration counts and active set size when using $P_{3,AMG}$. Given the maximum number of iteration counts for a given time step is only $11$, we conclude that the use of the inexact $AMG$ solver on the ${\cal K}$ block has little effect on proposed solver. 

\newpage

\subsection{Three Space Dimensions}

\subsubsection{Grain Coarsening}


\begin{table}[]
\begin{center}
\begin{tabular}{ | c || c | c |  c |   }
\hline
\multicolumn{4}{|c|}{Two Phases - 3D} \\
\hline
	 $\varepsilon$  & {Mesh 1 (17576)} & {Mesh 2 (29791)}  & {Mesh 3 (68921)}  \\ 
	 \hline
	 \hline
	0.04  & $3$  & $3$  & $3$      \\ 
	\hline
	0.02  & $3$  & $3$    & $3$      \\ 
	\hline
\hline
	\multicolumn{4}{|c|}{Four Phases - 3D} \\
	 \hline
	 \hline
	0.04  & $13/{\bf 7}$ & $11/{\bf 7}$ &   $10/{\bf 7}$   \\ 
	\hline
	0.02      &    $11/{\bf 6}$  &   $ 9/{\bf 6}$   &   $ 9/{\bf 6}$     \\ 
	\hline
\hline
\multicolumn{4}{|c|}{Six Phases - 3D} \\
	 \hline
	 \hline
	0.04   & $ 14/{\bf 9}$ &  $10/{\bf 7} $ & $ 8/{\bf 6} $   \\ 
	\hline
	0.02  &  $ 10/{\bf 5.5} $ &  $ 9/{\bf 7} $ & $ 9/{\bf 6.5}$ \\ 
	\hline
\end{tabular}
\caption{Maximum GMRES iteration counts when starting with a well mixed initial condition, using the exact preconditioner, ${P}_3$. (DOFS) denote the degree of freedom of each order parameter.}
 \label{tab: spin 3D} 
\end{center}
\end{table}
In Table \ref{tab: spin 3D} we present the GMRES iteration counts when using the preconditioner $P_3$. As in the two dimensional case, we see that there is little dependence on any of the parameters, mesh size, $\varepsilon$, or number of phases, and again we observe the three iteration convergence of the two phase problem. Turning to the fully practical preconditiponer, $P_{3,AMG}$, in Table \ref{tab: AMG 3D} we present CPU timings for this problem. It is not clear how to measure how these timings scale, since as the mesh is refined and more phases are added, the active set size changes considerably. However, we feel that these non-optimized CPU timings are an excellent indicator of the scalability of the proposed approach. 
\begin{table}[]
\begin{center}
\begin{tabular}{ | c || c | c |  c |   }
\hline
\multicolumn{4}{|c|}{ $\varepsilon = 0.04$} \\
\hline
	   & {Mesh 1 (9261)} & {Mesh 2 (29791)}  & {Mesh 3 (68921)}  \\ 
	 \hline
	 \hline
	$N = 2$  & $72s$  & $342s$  & $1133s$      \\ 
	\hline
	$N = 4$  & $135s$  & $824s$  & $2961s$      \\ 
	\hline
	$N = 6$  & $284s$  & $1324s$  & $ 8356s$      \\ 
	\hline
\end{tabular}
\caption{CPU timings using ${P}_{3,AMG}$ in three space dimensions. Initial condition is well mixed and $T=2$.}
 \label{tab: AMG 3D} 
\end{center}
\end{table}

Finally, for the well mixed problem we consider an initial problem of a well mixed sphere of $8$ phases surrounded by a final pure $9$-th phase, we take $\varepsilon=0.04$. The mesh used has over a half a million nodes, this leads to a system size of more than five million degrees of freedom. The evolution of these phases can be seen in Figure \ref{fig: 9 phase spin 3D}.
\begin{figure}[h]
  \centering
  \subfloat[Phases 1-4, $T=0$.]{\label{fig:Spin10_a0_3D}\includegraphics[width=0.33\textwidth]{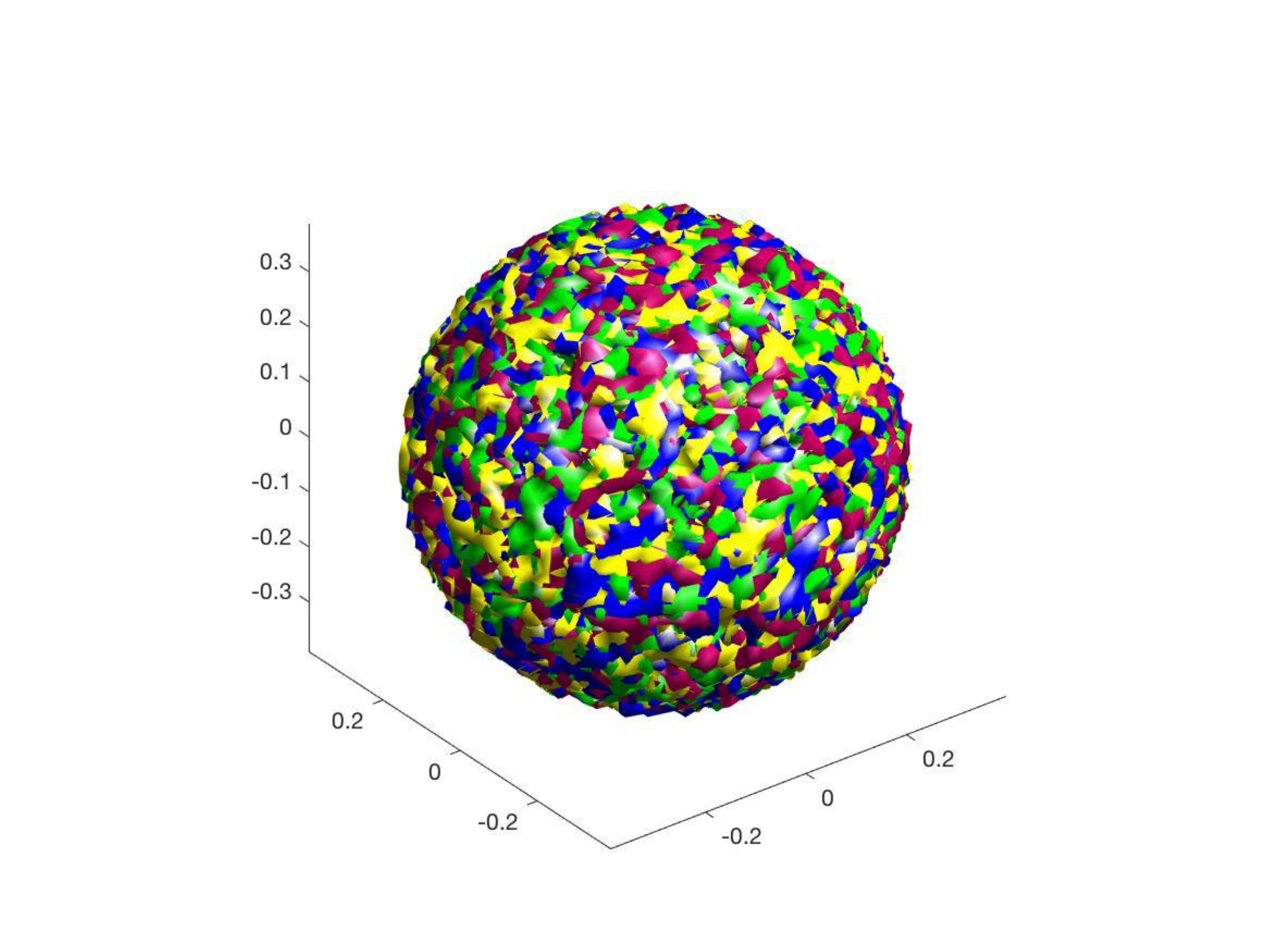}}
  \subfloat[Phases 5-8, $T=0$.]{\label{fig:Spin10_b0_3D}\includegraphics[width=0.33\textwidth]{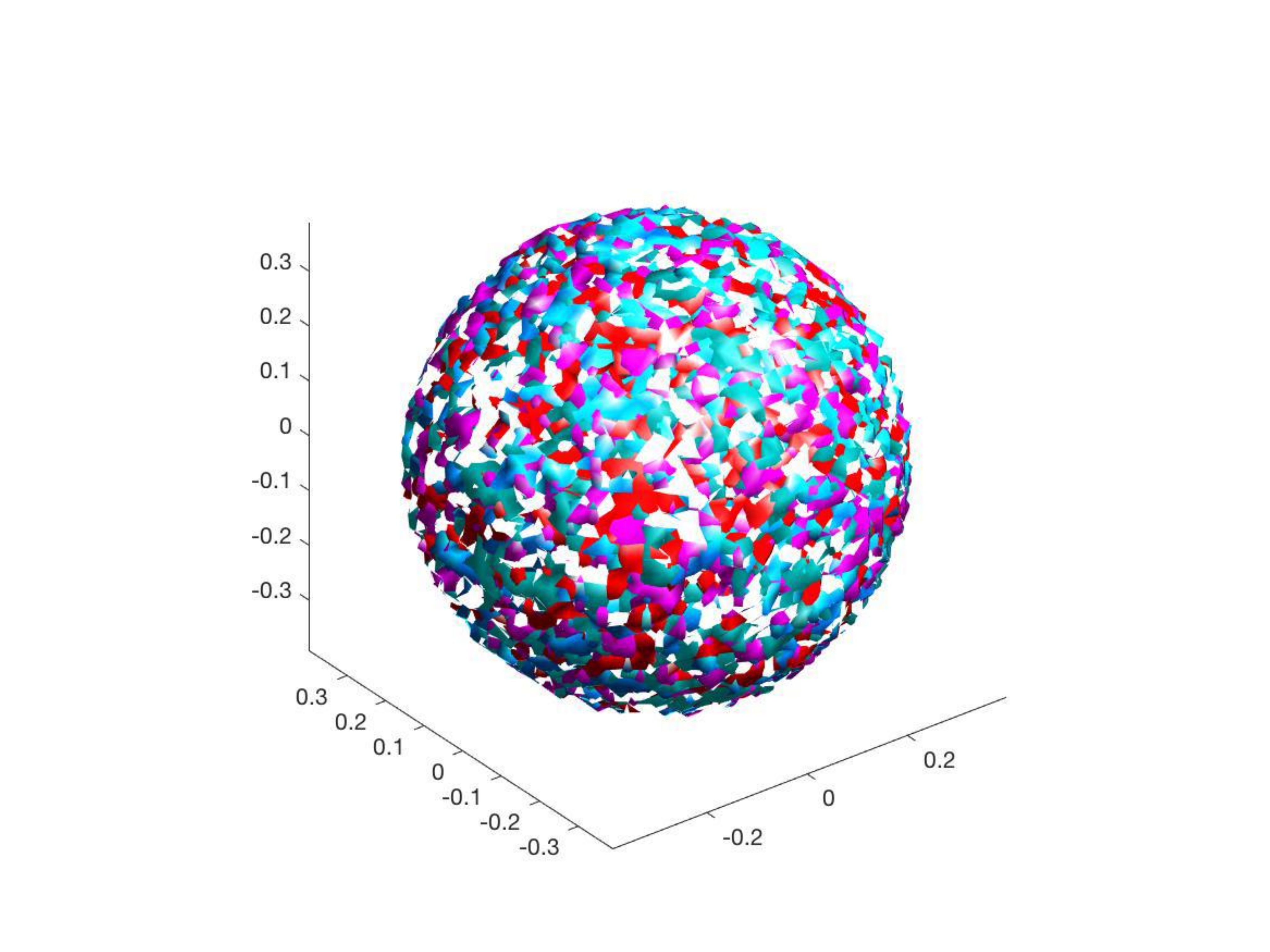}}\\
  \subfloat[Phases 1-4, $T=0.1$.]{\label{fig:Spin10_a1_3D}\includegraphics[width=0.33\textwidth]{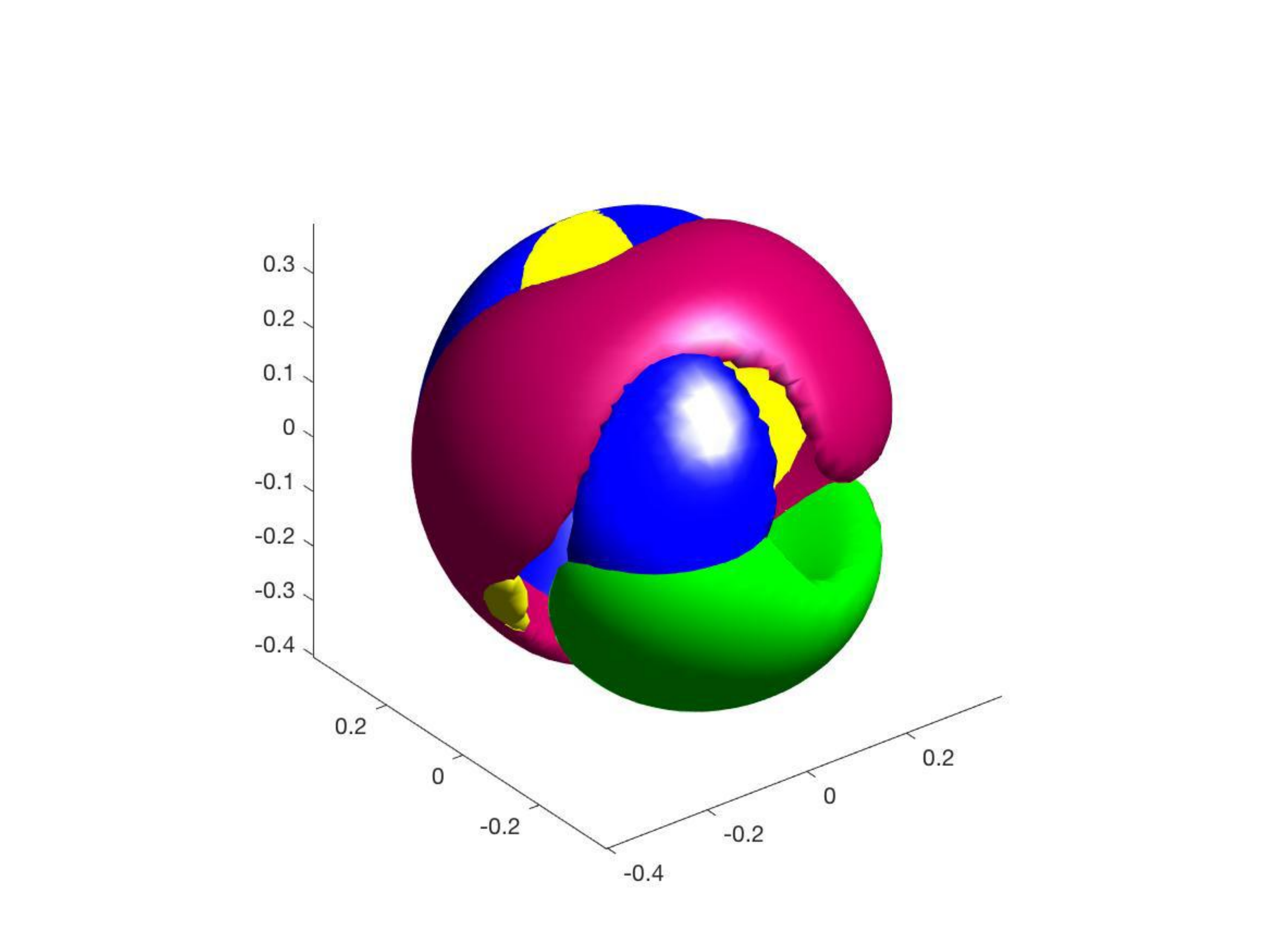}}
  \subfloat[Phases 1-4, $T=1$.]{\label{fig:Spin10_a2_3D}\includegraphics[width=0.33\textwidth]{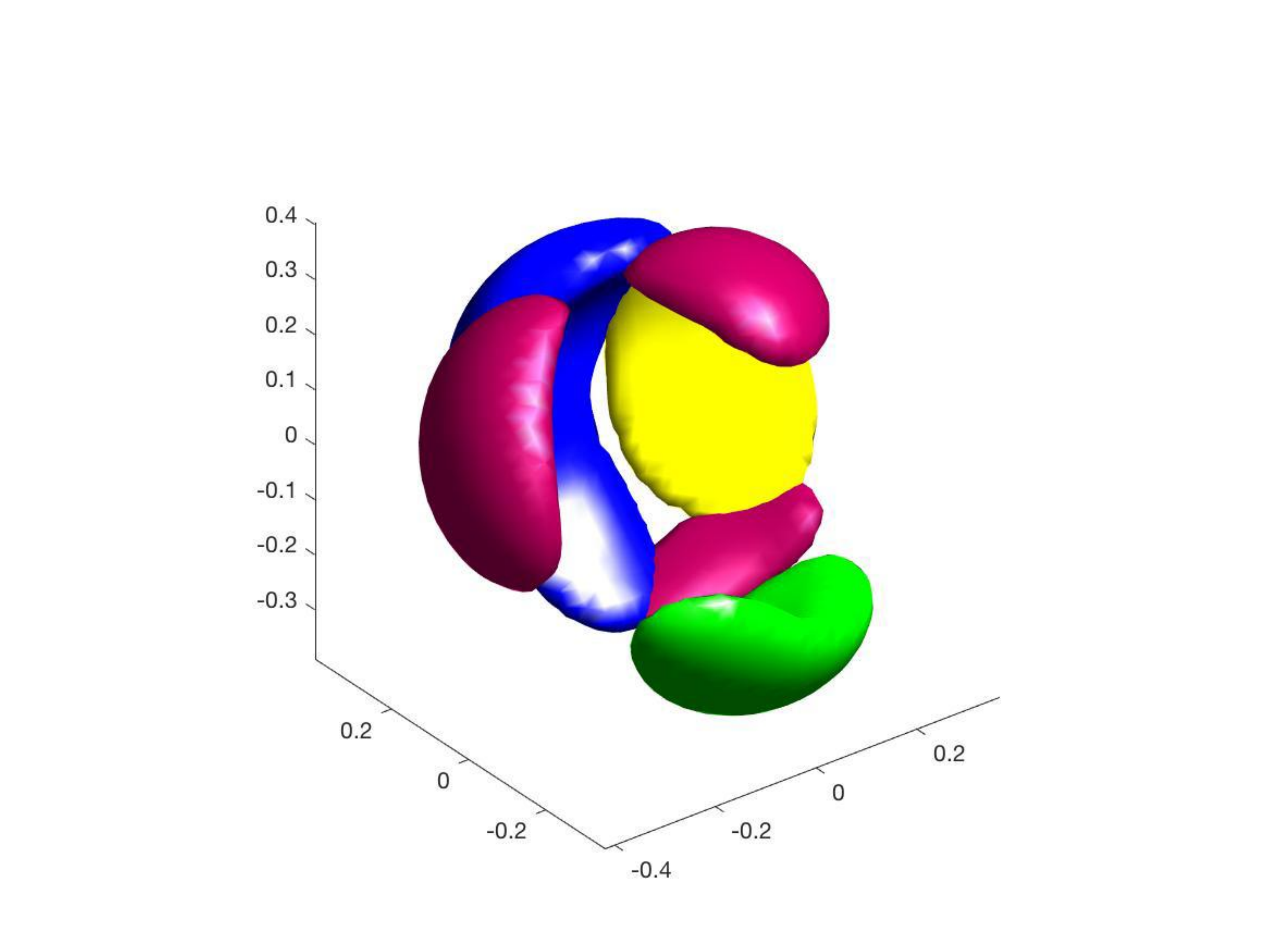}}
  \subfloat[Phases 1-4, $T=100$.]{\label{fig:Spin10_a3_3D}\includegraphics[width=0.33\textwidth]{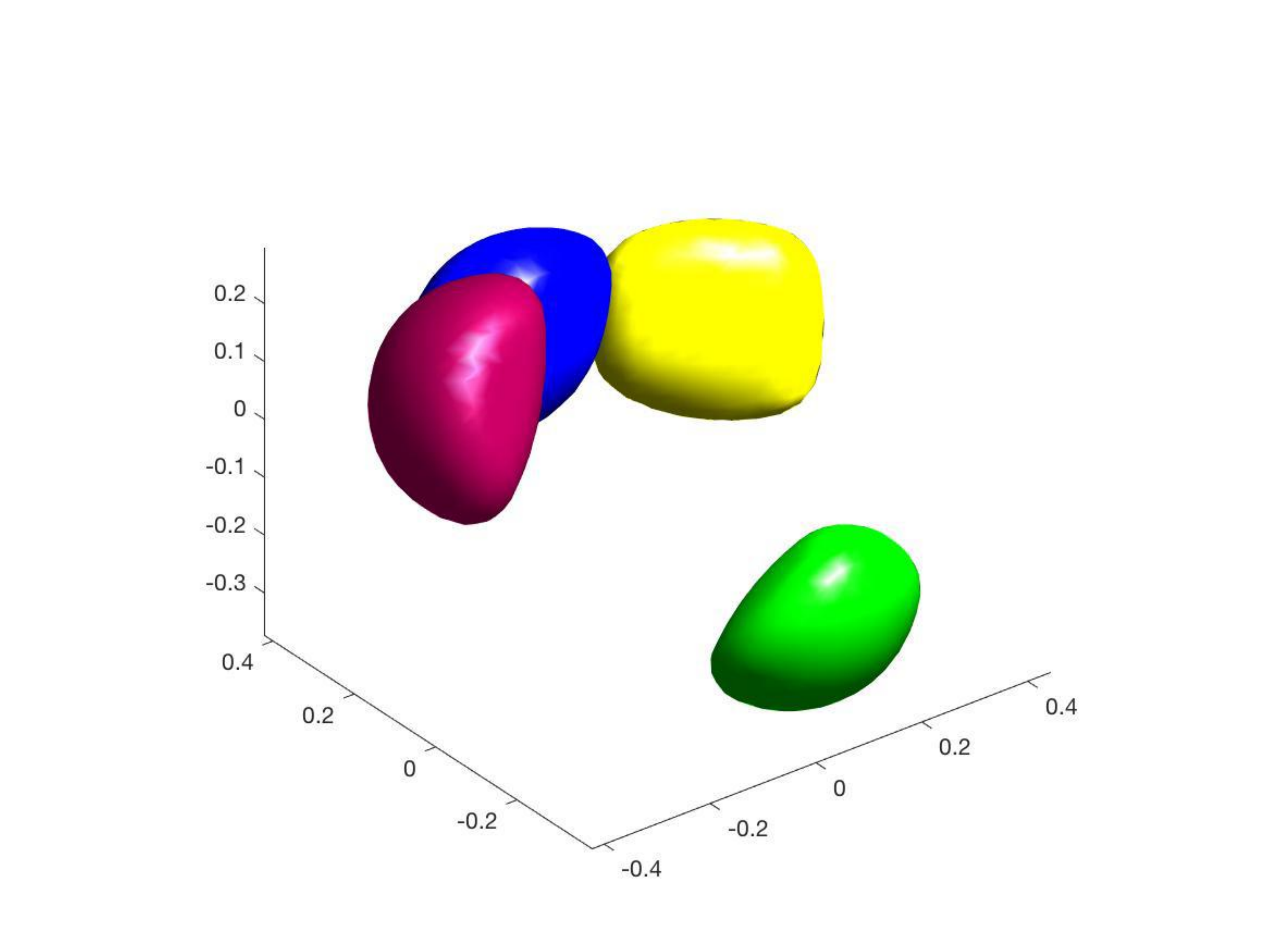}} \\
   \subfloat[Phases 5-8, $T=0.1$.]{\label{fig:Spin10_b1_3D}\includegraphics[width=0.33\textwidth]{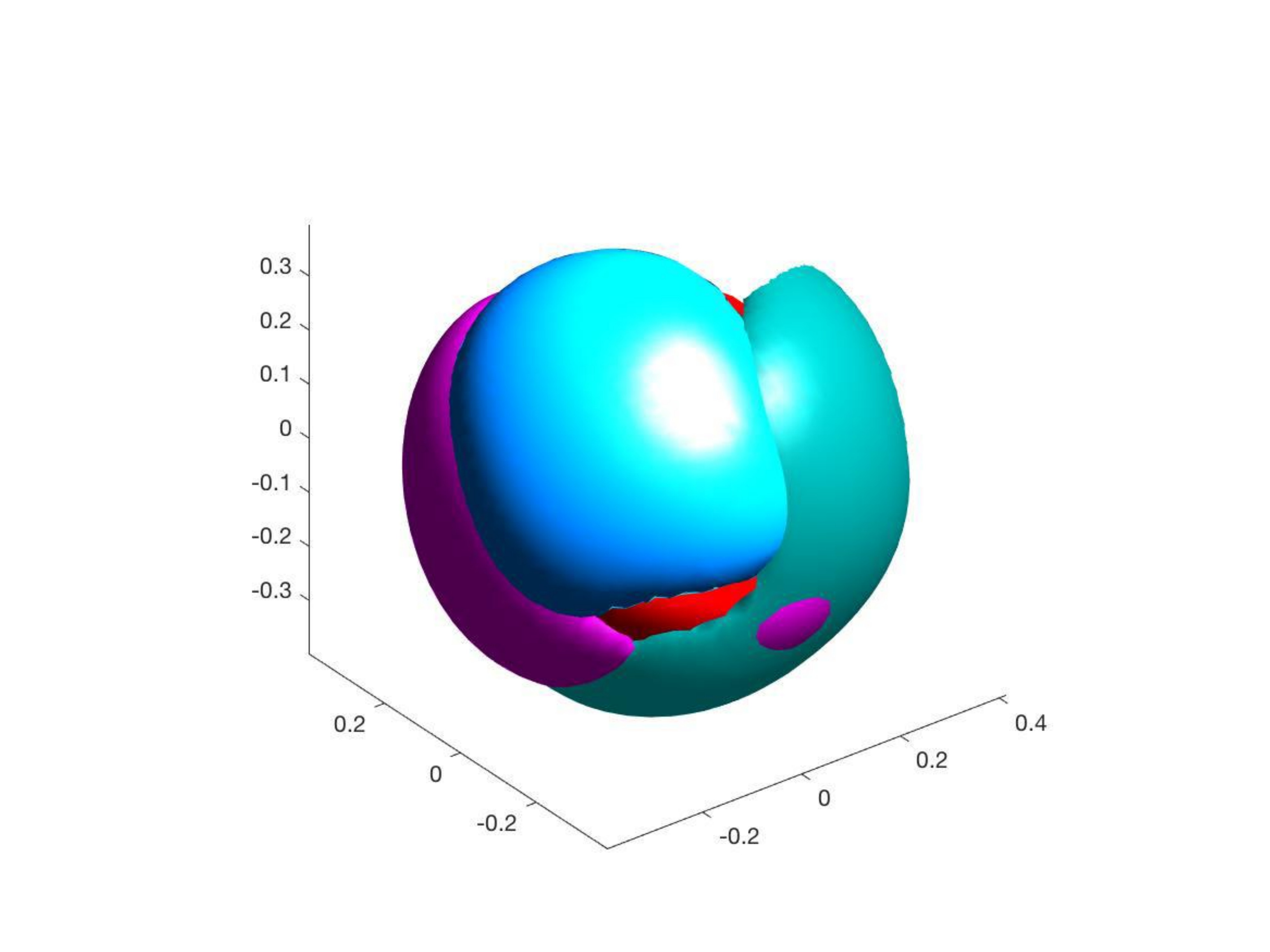}}
  \subfloat[Phases 5-8, $T=1$.]{\label{fig:Spin10_b2_3D}\includegraphics[width=0.33\textwidth]{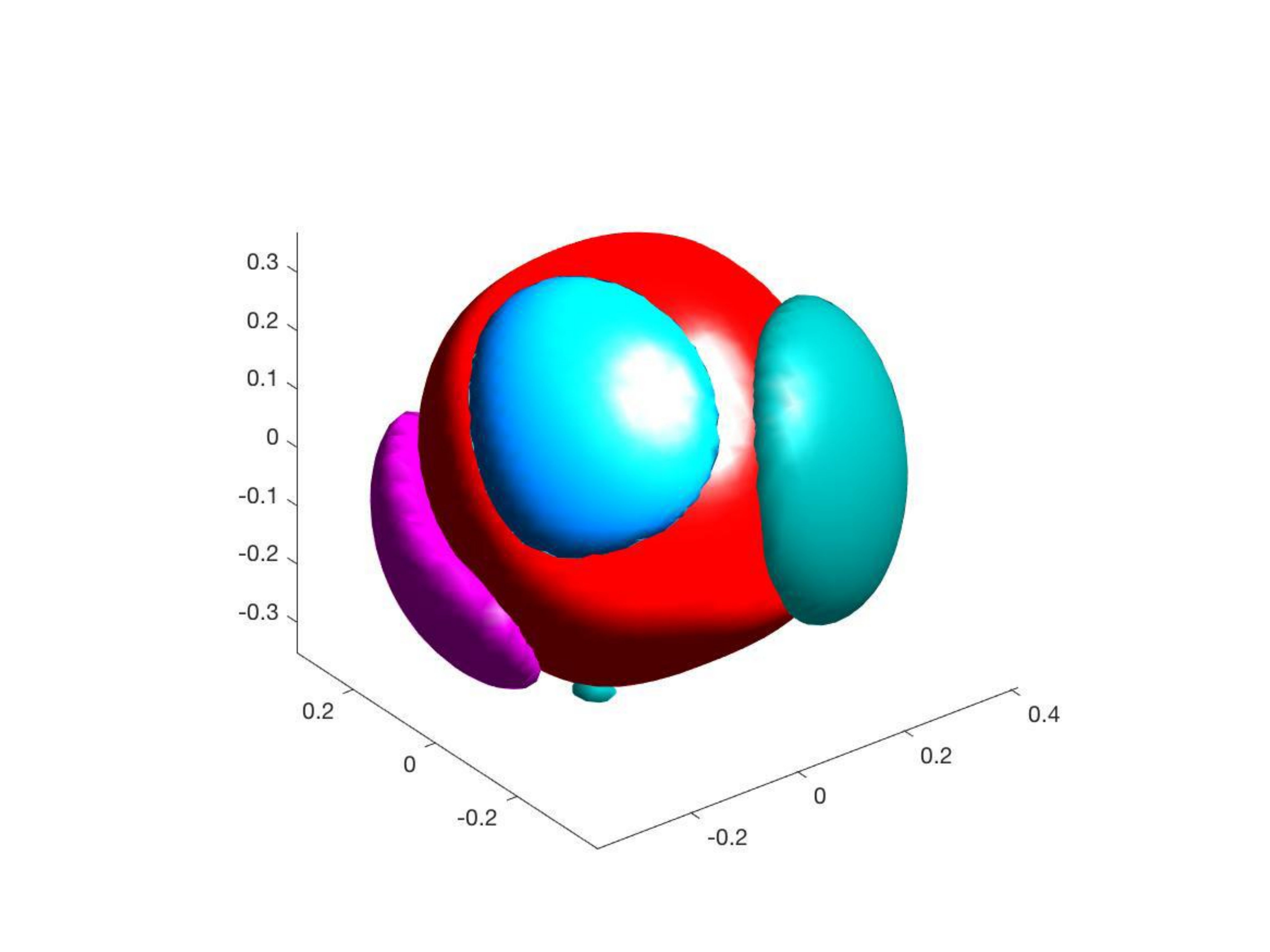}}
  \subfloat[Phases 5-8, $T=100$.]{\label{fig:Spin10_b3_3D}\includegraphics[width=0.33\textwidth]{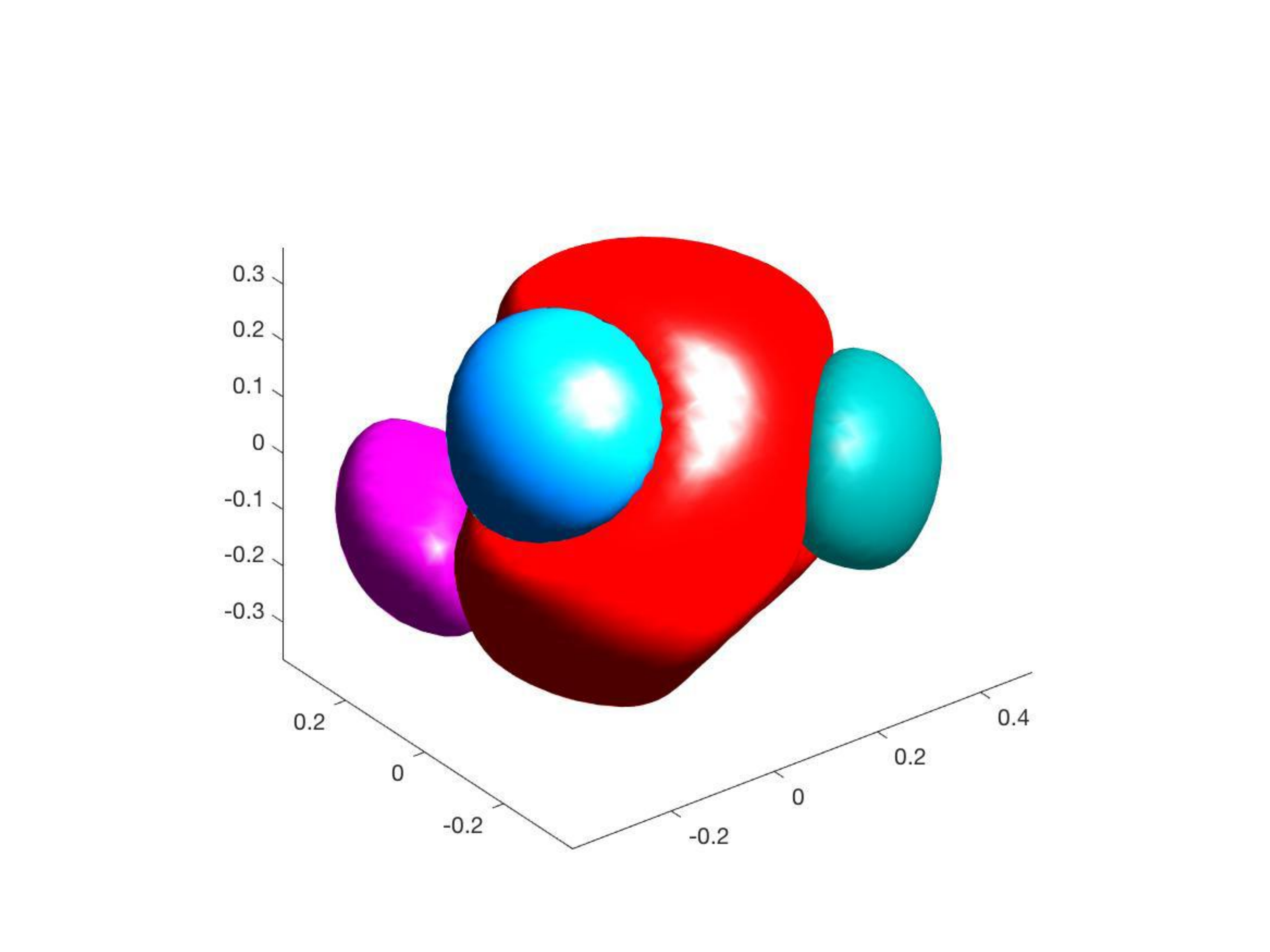}}
\caption{Initial condition is a central sphere containing eight well mixed phases, surrounded by a ninth pure phase.}
  \label{fig: 9 phase spin 3D}
\end{figure}
\begin{figure}[h]
  \centering
  \includegraphics[width=0.75\textwidth]{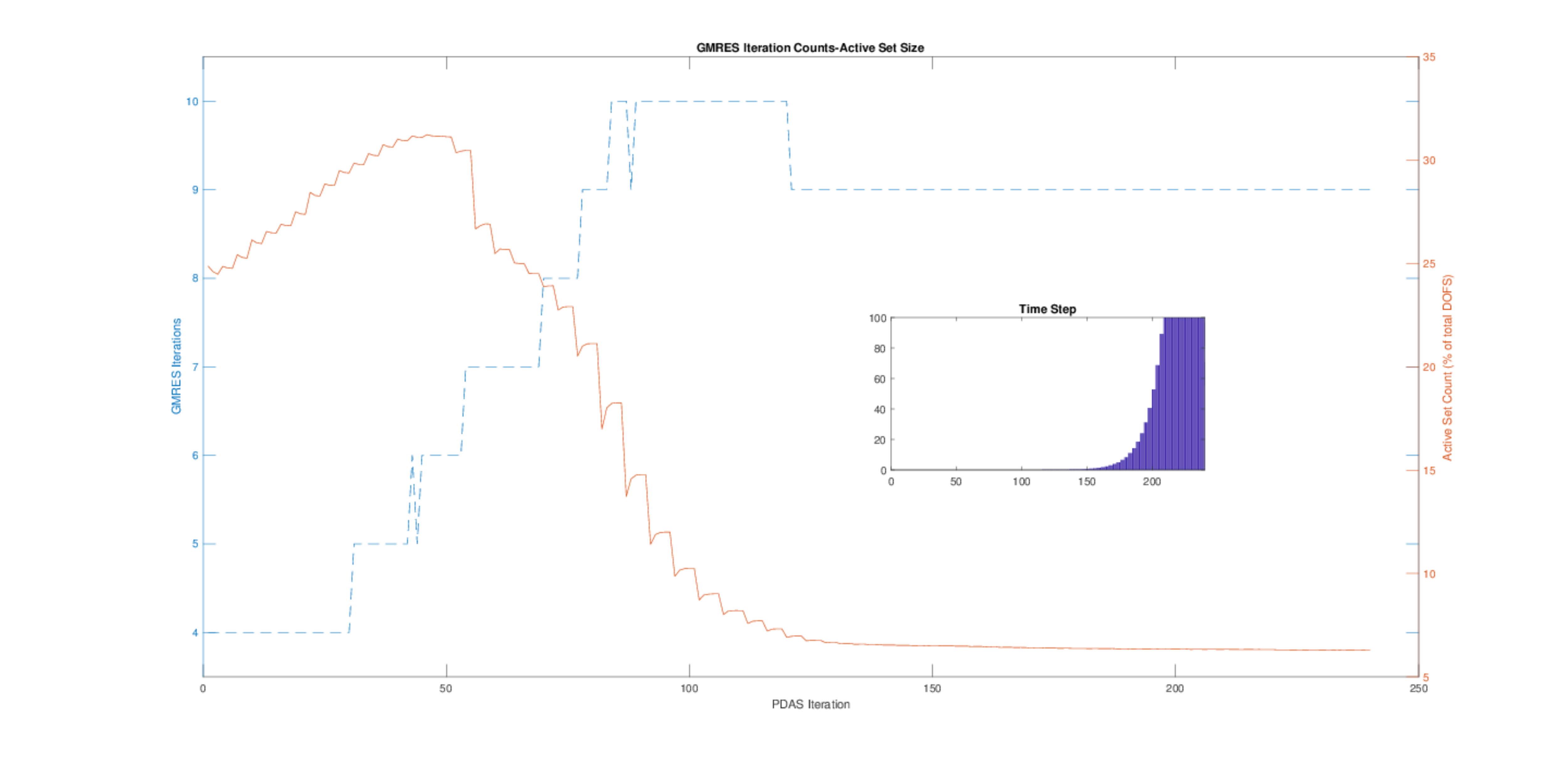}
  \caption{GMRES iteration counts using the preconditioner, ${P}_{3,AMG}$. Initial condition is a well mixed initial condition of nine phases in three space dimensions. (DOFS) denote the degree of freedom of each order parameter.}
  \label{fig: 9 phase spin 3D GMRES}
\end{figure}
Figure \ref{fig: 9 phase spin 3D GMRES} shows the iteration counts and active set size for each time step.

\subsubsection{Quadruple Junction to Triple Junction}

Finally, we consider a three dimensional problem analogous to the two dimensional quadruple junction problem. This consists of four bulk phases surrounded by a fifth phase, see Figure \ref{fig: 5 phase bulk 3D}. The initial mesh has over half a million nodes and $\varepsilon = 0.04$.  We see the evolution into bulk regions with spherical like minimal surfaces in contact with the fifth phase. Moreover, the central region shifts, in a similar way to the two dimensional problem, to remove any quadruple junctions. The final plot in Figure \ref{fig: 5 phase bulk 3D} shows the desired iteration counts for the solver. 
\begin{figure}[h]
  \centering
  \subfloat[$T=0$]{\label{fig:Bulk5_a_3D}\includegraphics[width=0.33\textwidth]{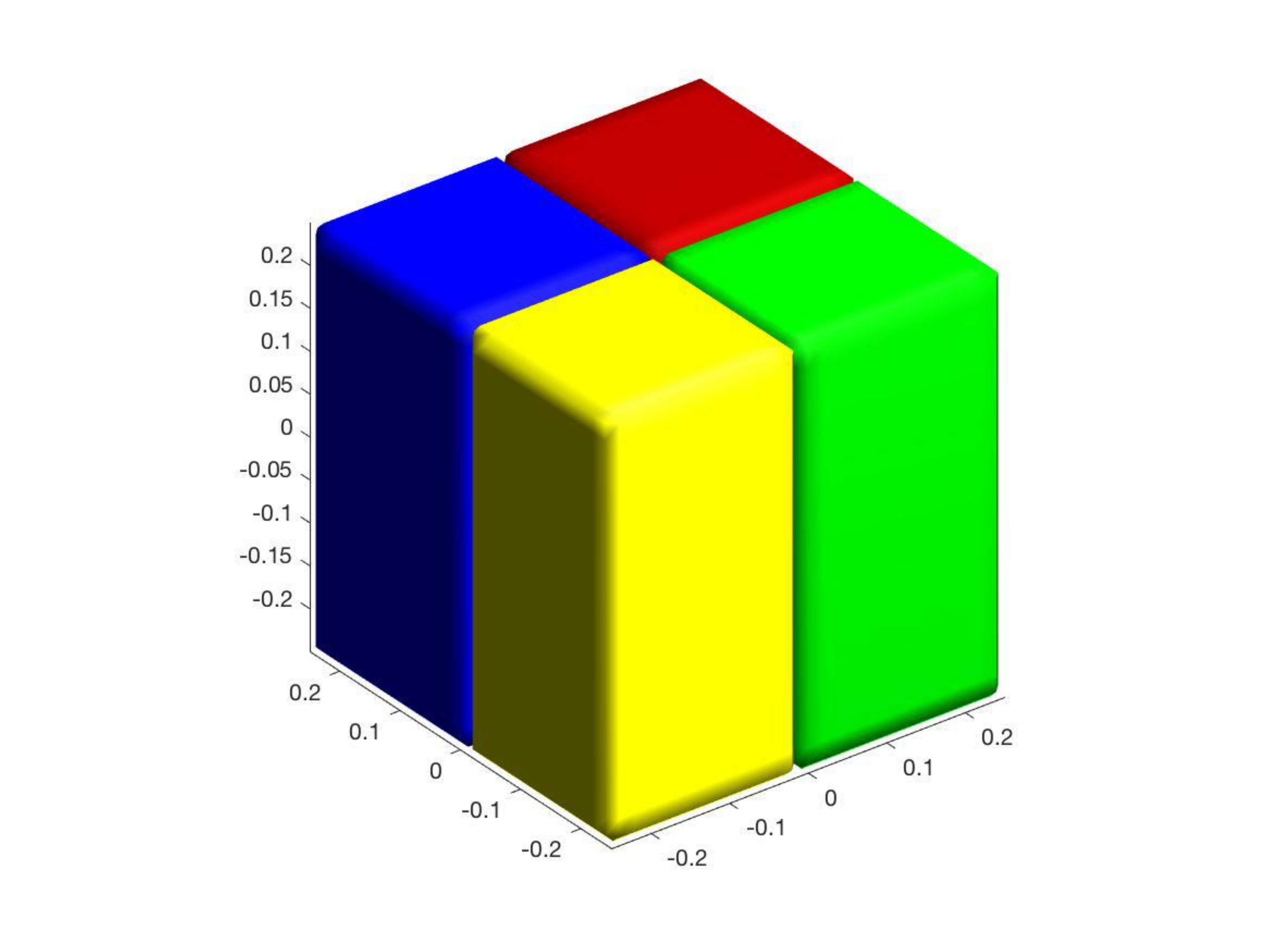}}
  \subfloat[$T=0.001$]{\label{fig:Bulk5_b_3D}\includegraphics[width=0.33\textwidth]{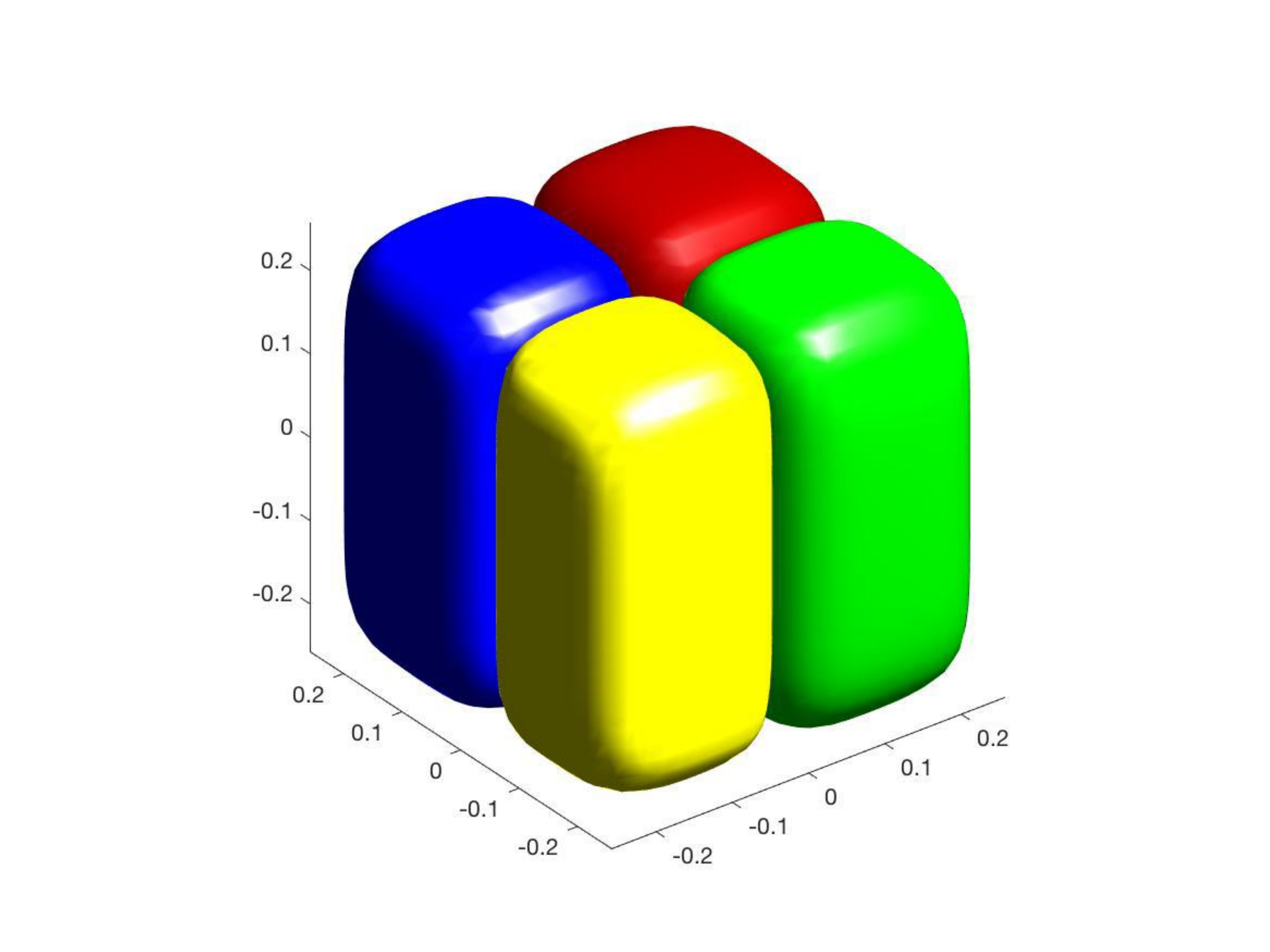}}
  \subfloat[$T=1$]{\label{fig:Bulk5_c_3D}\includegraphics[width=0.33\textwidth]{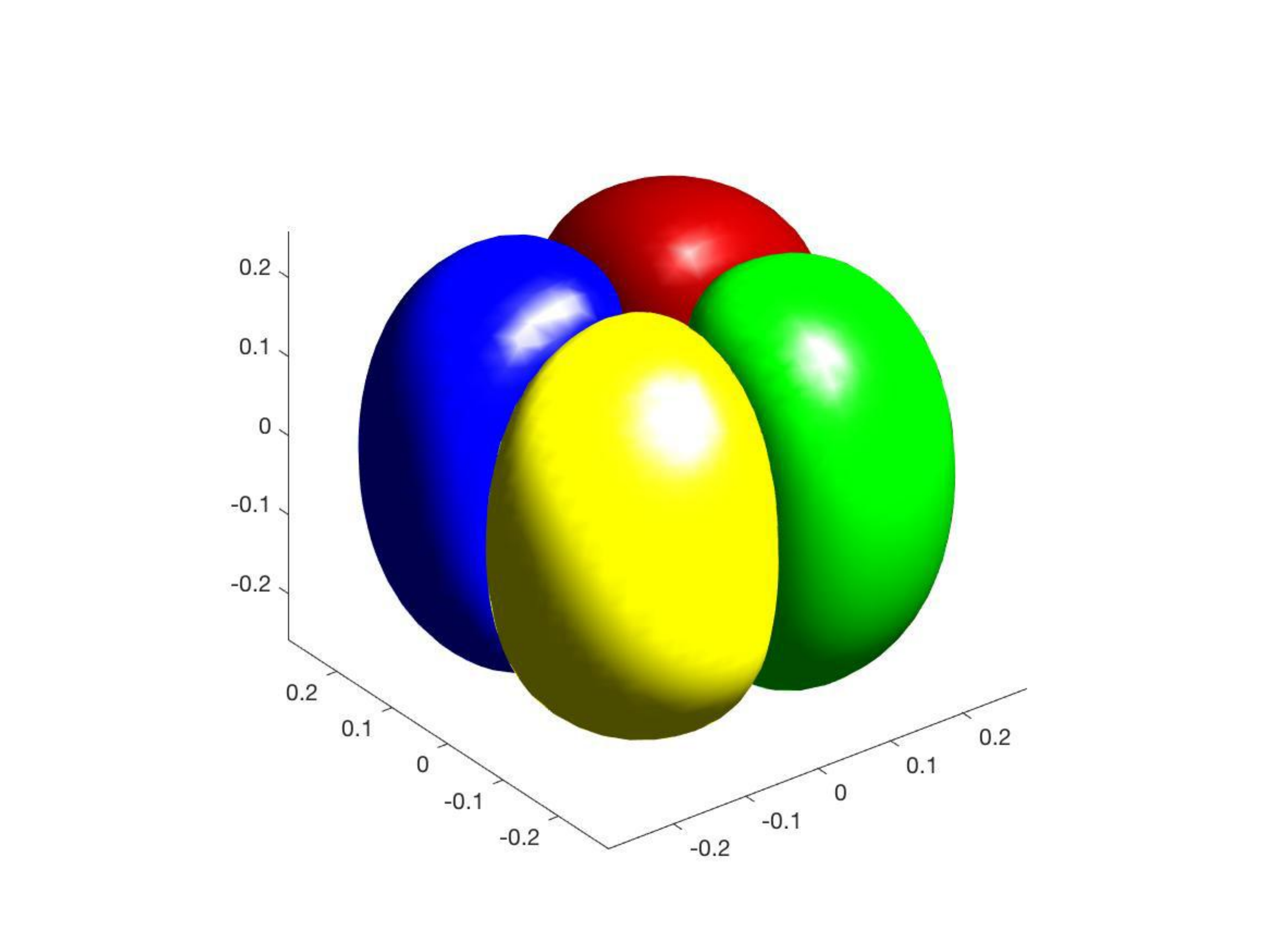}} \\
  \subfloat[$T=100$]{\label{fig:Bulk5_d_3D}\includegraphics[width=0.33\textwidth]{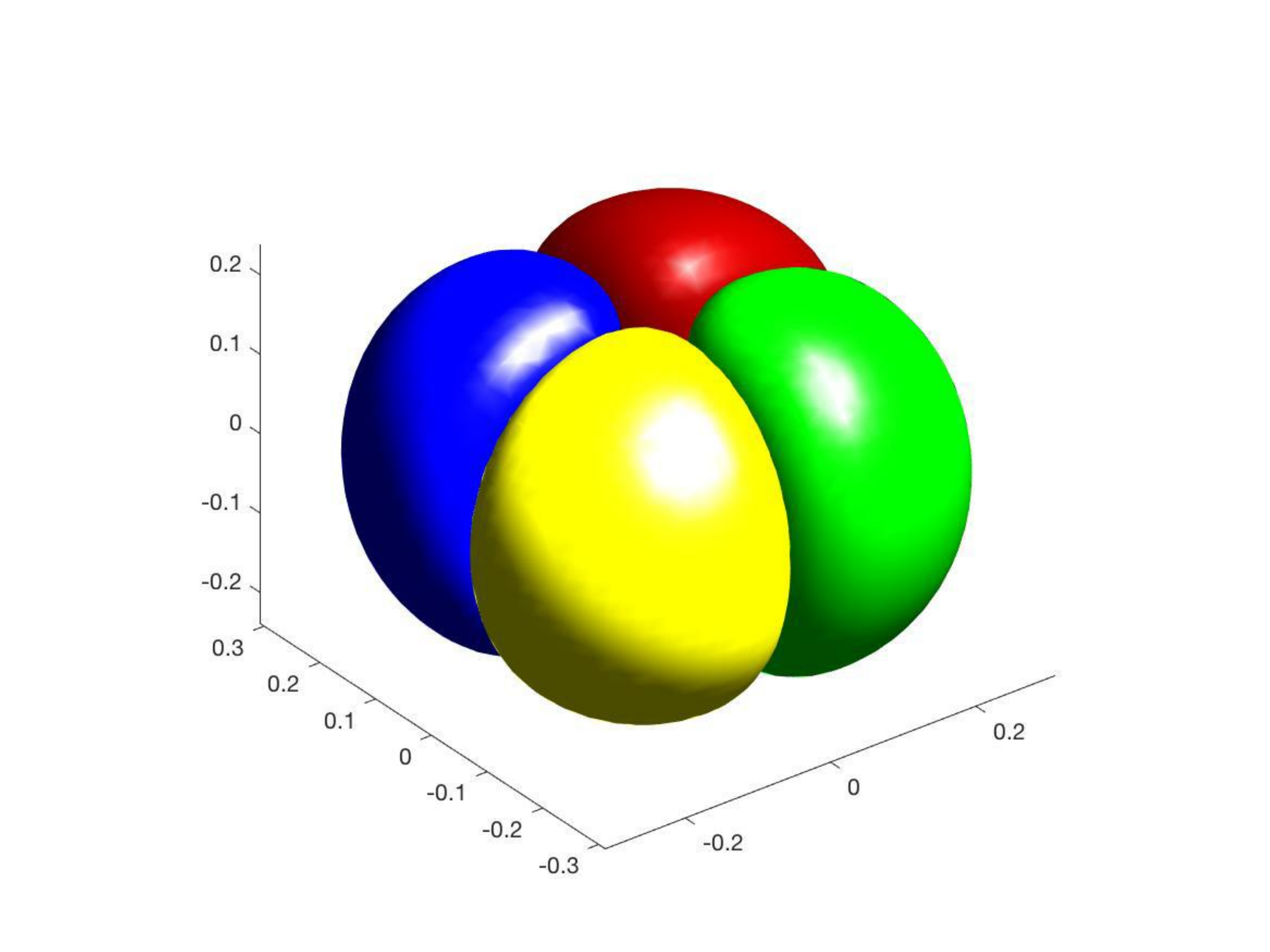}}
  \subfloat[$T=5000$]{\label{fig:Bulk5_e_3D}\includegraphics[width=0.33\textwidth]{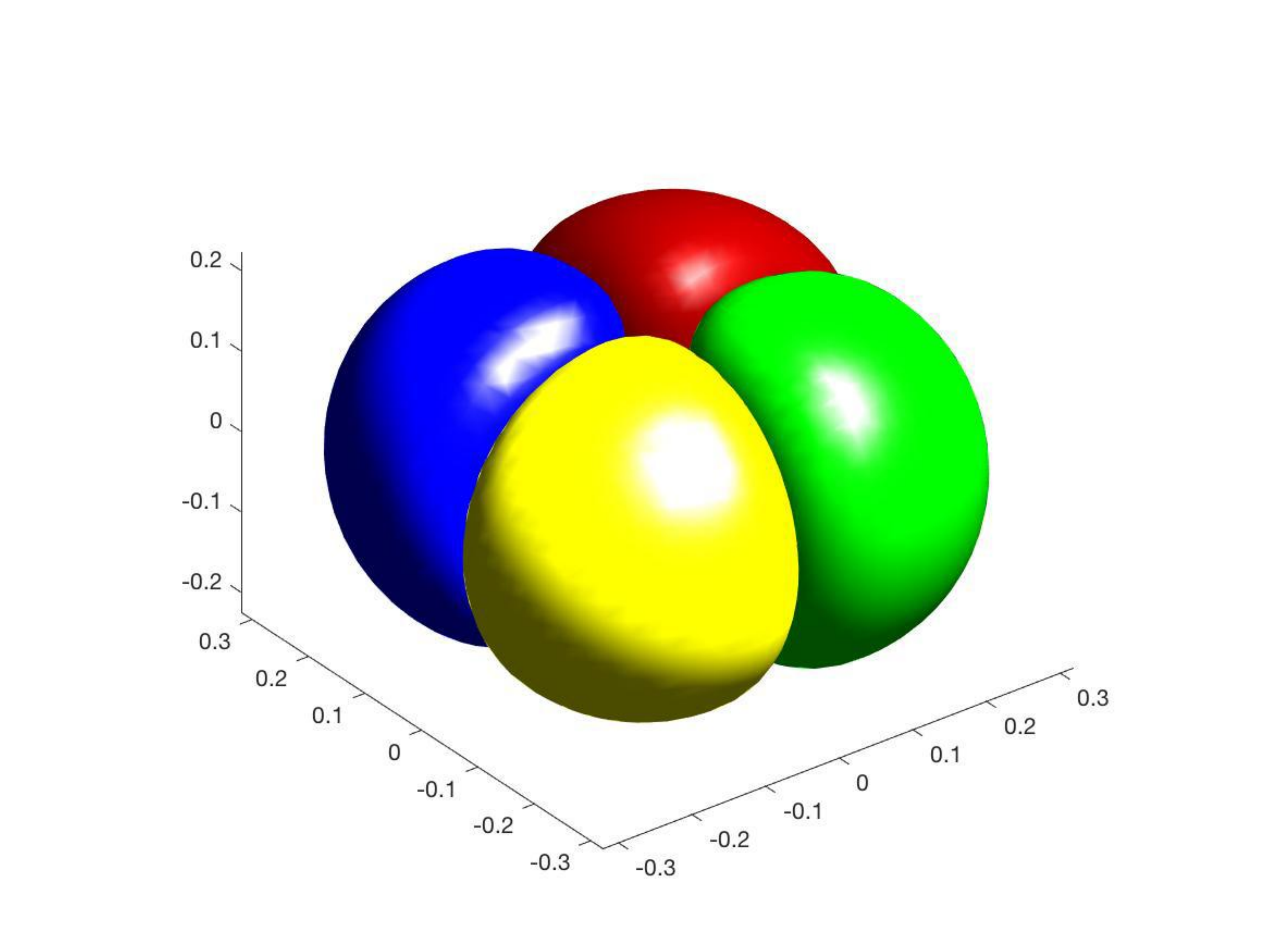}}
  \subfloat[GMRES Iterations for $P_{3,AMG}$ and active set size]{\label{fig:Bulk5_f_3D}\includegraphics[width=0.33\textwidth]{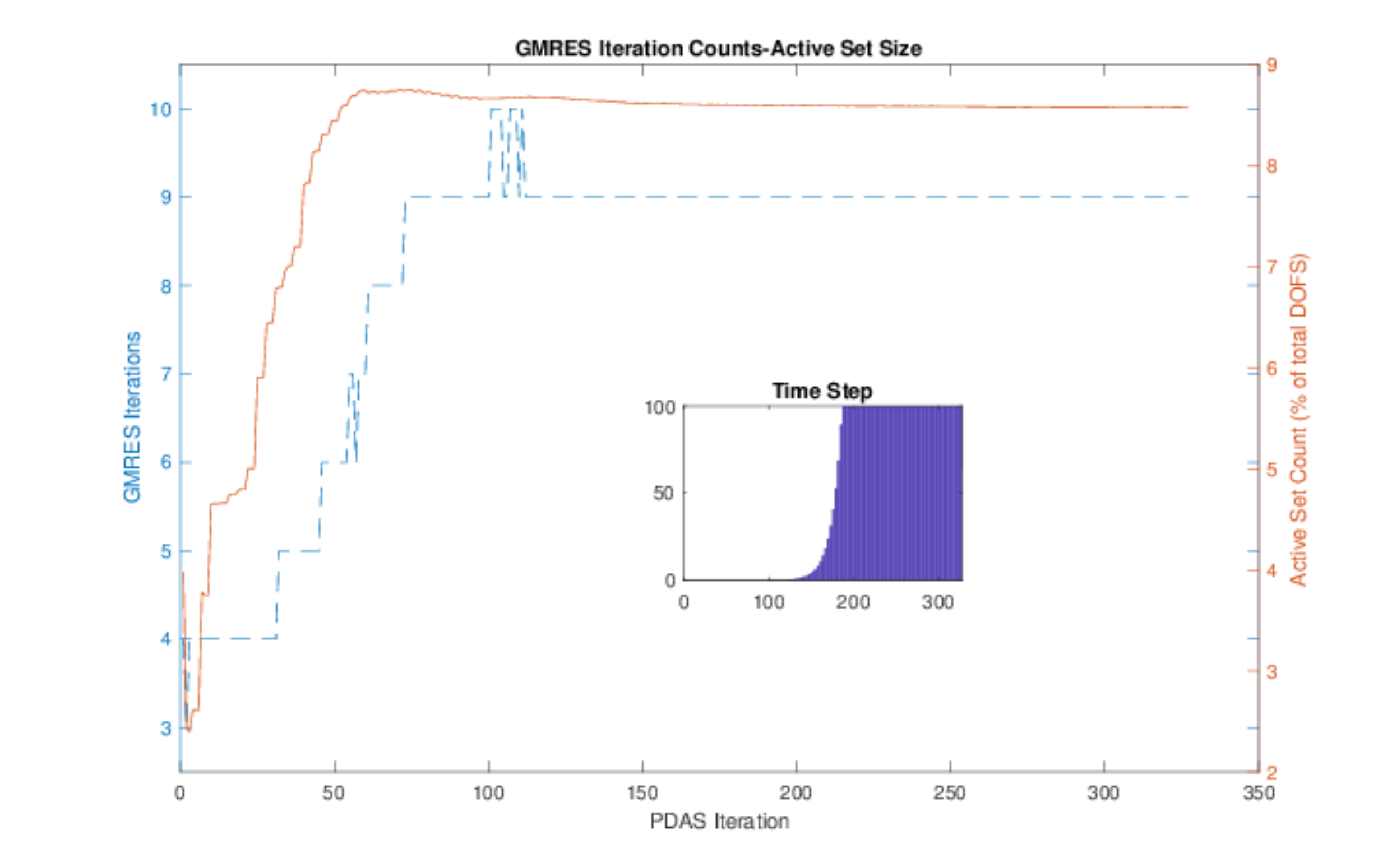}}
\caption{Time evolution of four bulk phases surrounded by a fifth phase.}
  \label{fig: 5 phase bulk 3D}
\end{figure}

\section{Conclusions}

In this work we have presented a robust practical preconditioner for systems of multiphase Allen-Cahn variational inequalities. As mentioned earlier, see Remark \ref{tolerance remark}, the need for a reliable and efficient solver is crucial when using iterative methods to solve the linear systems arising in the PDAS algorithm, where solve tolerances have to be small. Firstly, when exactly solving the matrices in the preconditioning system, in the case of two phases, it was shown, in Theorem \ref{theorem}, that GMRES will converge within three iterations. Secondly, in the case of multiple phases, it was shown experimentally, that the use of the precondioner $P_{3,AMG}$ leads to low GMRES iteration counts on fine meshes. Finally, given the standard blocks used in this solver, i.e., Multigrid, GMRES and simple smoothers, the proposed approach may immediately be applied in most of the software packages used to solve multiphase variational inequalities.  

\noindent {\bf Acknowledgements}\\
VS would like to thank the Isaac Newton Institute for Mathematical Sciences for support and hospitality during the programme {\it Geometry, compatibility and structure preservation in computational differential equations} when work on this paper was undertaken. 
This work was supported by: EPSRC grant number EP/R014604/1.  


\end{document}